\documentclass[a4paper,12pt]{article}
\usepackage{amsmath,amsthm,amsfonts,amssymb,enumitem,graphicx,float,calc,microtype,mathtools,thmtools,underscore,anyfontsize,todonotes,fontenc,lmodern}
\usepackage{latexsym}
\usepackage{tikz}
\usepackage{xcolor}
\usepackage{paralist}
\usepackage{enumitem}
\usepackage[english]{babel}
\usepackage[lmargin=30mm,rmargin=30mm,tmargin=30mm,bmargin=30mm]{geometry}
\renewcommand{\baselinestretch}{1.15}
\allowdisplaybreaks

\usepackage[utf8]{inputenc}
\usepackage[T1]{fontenc}
\usepackage{xcolor}

\usepackage[numbers,sort&compress,longnamesfirst]{natbib}
\def\NAT@spacechar{~}
\makeatother

\usepackage[pdftitle={??},
pdfauthor={Hickingbotham -- Wood},
colorlinks =true,
linkcolor={black},
urlcolor={blue!80!black},
filecolor={blue!80!black},
citecolor={black},
anchorcolor=blue,
filecolor=blue, 
menucolor=blue]{hyperref}

\usepackage[noabbrev,capitalise]{cleveref}
\crefname{lem}{Lemma}{Lemmas}
\crefname{thm}{Theorem}{Theorems}
\crefname{cor}{Corollary}{Corollaries}
\crefname{prop}{Proposition}{Propositions}
\crefname{conj}{Conjecture}{Conjectures}
\crefname{open}{Open Problem}{Open Problems}
\crefname{obs}{Observation}{Observations}

\crefformat{equation}{(#2#1#3)}
\Crefformat{equation}{Equation #2(#1)#3}

\theoremstyle{plain}
\newtheorem{thm}{Theorem}
\newtheorem{lem}[thm]{Lemma}
\newtheorem{cor}[thm]{Corollary}
\newtheorem{obs}[thm]{Observation}
\newtheorem{prop}[thm]{Proposition}

\theoremstyle{definition}

\DeclarePairedDelimiter{\ceil}{\lceil}{\rceil}
\DeclarePairedDelimiter{\floor}{\lfloor}{\rfloor}

\DeclareMathOperator{\tw}{tw}
\DeclareMathOperator*{\pw}{pw}

\DeclareMathOperator{\degen}{degen}

\renewcommand{\leq}{\leqslant}

\renewcommand{\geq}{\geqslant}

\newcommand{\CartProd}{\mathbin{\square}}

\theoremstyle{definition}

\DeclareMathOperator*{\dist}{dist}

\DeclareMathOperator*{\bs}{\backslash}
\DeclareMathOperator*{\sse}{\subseteq}
\DeclareMathOperator*{\R}{\mathbb{R}}
\DeclareMathOperator*{\N}{\mathbb{N}}

\DeclareMathOperator*{\G}{\mathcal{G}}
\DeclareMathOperator*{\Path}{\mathcal{P}}
\DeclareMathOperator*{\Pcal}{\mathcal{P}}

\DeclareMathOperator*{\B}{\mathcal{B}}

\DeclareMathOperator*{\W}{\mathcal{W}}

\DeclareMathOperator*{\pathnumber}{path}

\DeclareMathOperator*{\dll}{dll}

\title{\bf Structural Properties of Graph Products}
 \author{%
 	Robert Hickingbotham,\!\!%
 	\thanks{School of Mathematics, Monash University, Melbourne, Australia (\texttt{robert.hickingbotham@monash.edu}). Research supported by an Australian Government Research Training Program Scholarship.}
 	\,\,
 	David R. Wood\thanks{School of Mathematics, Monash University, Melbourne, Australia (\texttt{david.wood@monash.edu}). Research supported by the Australian Research Council.}
 }

\begin{document}
\maketitle

\begin{abstract}
	 Dujmovi\'{c}, Joret, Micek, Morin, Ueckerdt, and Wood [J. ACM 2020] established that every planar graph is a subgraph of the strong product of a graph with bounded treewidth and a path. Motivated by this result, this paper systematically studies various structural properties of cartesian, direct and strong products. In particular, we characterise when these graph products contain a given complete multipartite subgraph, determine tight bounds for their degeneracy, establish new lower bounds for the treewidth of cartesian and strong products, and characterise when they have bounded treewidth and when they have bounded pathwidth.
\end{abstract}

\section{Introduction}
\citet{DJMMUW20} recently showed that every planar graph is a subgraph of the strong product of a graph with bounded treewidth and a path.
\begin{thm}[\cite{DJMMUW20}]\label{GPST}
	Every planar graph is isomorphic to a subgraph of:
	\begin{compactitem}
		\item $H \boxtimes P$ for some graph $H$ of treewidth at most $8$ and for some path $P$;
		\item $H \boxtimes P \boxtimes K_3$ for some graph $H$ of treewidth at most $3$ and for some path $P$.
	\end{compactitem}
\end{thm}
This breakthrough has been the key tool to resolve several major open problems regarding queue layouts \cite{DJMMUW20}, nonrepetitive colourings \cite{DEJWW2020nonrepetitive}, centred colourings \cite{debski2020improved}, and adjacency labelling schemes \cite{bonamy2020shorter,EJM2020universal,DEGJMM2020adjacency}. Treewidth measures how similar a graph is to a tree and is an important parameter in algorithmic and structural graph theory (see \cite{bodlaender1993tourist,HW2017tied,reed1997treewidth}). Graphs with bounded treewidth are a simpler class of graphs compared to planar graphs. \Cref{GPST} therefore reduces problems on a complicated class of graphs (planar graphs) to a simpler class of graphs (bounded treewidth). \Citet{UWY2021GPST} improved \Cref{GPST} by replacing ``treewidth at most 8'' with ``simple treewidth at most 6''. Analogous \emph{graph product structure theorems} hold for other graph classes such as graphs with bounded genus \cite{DJMMUW20}, minor-closed graph classes \cite{DJMMUW20}, map graphs \cite{DMW20}, $k$-planar graphs \cite{DMW20}, and infinite graphs \cite{HMSTW2021universality}. 

These results motivate further study into the structural properties of graph products. In this paper, we explore the following properties of cartesian, direct and strong products: complete multipartite subgraphs, degeneracy, pathwidth and treewidth. Our key contributions are the following:

\textbf{Complete multipartite subgraphs:} We characterise the presence of complete multipartite subgraphs in cartesian, direct and strong products. These results are presented in \Cref{SectionMulti}.

\textbf{Degeneracy:} A graph $G$ is \emph{$d$-degenerate} if every subgraph of $G$ has minimum degree at most $d$. The \emph{degeneracy} of $G$, $\degen(G)$, is the minimum integer $d$ such $G$ is $d$-degenerate. Degeneracy is an important graph parameter as it measures the sparsity of a graph. Moreover, $d$-degenerate graphs are $(d+1)$-colourable and in fact, $(d+1)$-choosable. We present tight upper and lower bounds on the degeneracy of direct and strong products. See \Cref{SectionDegen} for these results.

\textbf{Pathwidth and treewidth of cartesian and strong products:} We establish two new lower bounds for the treewidth of cartesian and strong products. The first bound states that if one graph does not admit small $\epsilon$-separations and the other graph is connected with many vertices, then the cartesian product has large treewidth. The second bound states that if one graph has large treewidth and the other graph has large Hadwiger number, then the strong product has large treewidth. In addition, we characterise when the cartesian and strong product of two monotone graph classes has bounded treewidth and when it has bounded pathwidth. These results are presented in \Cref{SectiontwCart}.
 
\textbf{Pathwidth and treewidth of direct products:} We characterise when the direct product of two monotone graph classes has bounded treewidth and when it has bounded pathwidth. For treewidth, the characterisation states that the direct product of two graph classes has bounded treewidth if and only if the connected graphs in one of the classes have a bounded number of vertices while the graphs in the other class have bounded treewidth; or if the connected graphs in one of the classes have bounded vertex cover number while the graphs in the other class have bounded treewidth and bounded maximum degree. For pathwidth, our characterisation is directly analogous to that for treewidth with the stronger condition that the second class has bounded pathwidth. We also demonstrate that the treewidth of a graph is polynomially tied to the treewidth of the direct product of the graph with $K_2$. To our knowledge, it was previously open whether these two parameters were tied. These results are presented in \Cref{SectionDirectK2,SectionDirectProduct}. 
 
This line of research has been previously explored for the following properties of graph products: connectivity \cite{BS2007connectivity, BS2008connectivity, XY2007connectivity, Spacapan2008connectivity}; queue-number \cite{wood2005queue}; stack-number \cite{dujmovic2020stack,pupyrev2020book}; thinness \cite{bonomo2020thinness}; boxicity and cubicity \cite{CIMR2015boxicity}; polynomial growth \cite{dvovrak2020notes}; bounded expansion and colouring numbers \cite{dvovrak2020notes}; chromatic number \cite{sabidussi1964vertex, DSW1984chromatic, klavzar1996coloring, Shitov2019counterexamples, tardif2008hedetniemi, zhu2021hedetniemi,Vesztergombi78,Vesztergombi81}; and Hadwiger number \cite{ABPS1997product, CS2007conjecture, ivanco1988results, kotlov2001strong, miller1978contractions,zelinka1976finite, wood2011clique, pecaninovic2019complete, HW2021hadwiger}. Our companion paper \cite{HW2021hadwiger} studies the Hadwiger number of direct products. See the handbook by \citet{hammack2011handbook} for an in-depth treatment of graph products.

\subsection{Preliminaries}
For every integer $k \geq 1$, let $[k]:=\{1,\dots,k\}$. A \emph{graph parameter} is a function $\beta$ such that $\beta(G)\in\mathbb{R}$ for every graph $G$ and $\beta(G_1)=\beta(G_2)$ for all isomorphic graphs $G_1$ and $G_2$. Two graph parameters $\alpha$ and $\beta$ are \emph{tied }if there exists a function $f$ such that $\alpha(G)\leq f(\beta(G))$ and $\beta(G)\leq f(\alpha(G))$. If the function $f$ is polynomial, then they are \emph{polynomially tied}. 

Let $H$ and $G$ be graphs. $H$ is a \emph{minor} of $G$ if a graph isomorphic to $H$ can be obtained from $G$ by vertex deletion, edge deletion, and edge contraction. A \emph{model} of $H$ in $G$ is a function $\mu$ from $V(H)$ to disjoint connected subgraphs in $G$ such that $\mu(u)$ and $\mu(v)$ are adjacent whenever $uv \in E(H)$. It is folklore that $H$ is a minor of $G$ if and only if $G$ contains a model of $H$. Let $\textsf{v}(G):=|V(G)|$ be the \emph{order of $G$}. Let $\tilde{\textsf{v}}(G)$ be the maximum order of a connected component of $G$. For $v \in V(G)$, let $N_G(v):=\{u\in V(G):uv \in E(G)\}$ and $N_G[v]:=N(v)\cup \{v\}$. When the graph $G$ is clear, we drop the subscript $G$ and use the notation $N(v)$ and $N[v]$. Let $\Delta(G)$ and $\delta(G)$ respectively denote the maximum and minimum degree of $G$.

A \emph{class} of graphs $\G$ is a set of graphs that is closed under isomorphism and contains a graph with a non-empty vertex set. $\G$ is \emph{hereditary} if it is closed under induced subgraphs, \emph{monotone} if it is closed under subgraphs, and \emph{minor-closed} if it is closed under minors. For a graph parameter $\beta$, let $\beta(\G):=\sup\{\beta(G):G \in \G\}$. We say that $\beta(\G)$ is \emph{bounded} if there exists $c \in \R$ such that $\beta(G)\leq c$ for every graph $G \in \G$, otherwise we say that $\beta(\G)$ is \emph{unbounded}.

Let $G_1$ and $G_2$ be graphs. A \emph{graph product }$G_1 \bullet G_2$ is defined with vertex set: 
\begin{equation*}
	V(G_1 \bullet G_2):= \{(a,v): a \in V(G_1), v \in V(G_2)\}.
\end{equation*}

The \emph{cartesian product} $G_1 \CartProd G_2$ consists of edges of the form $(a,v)(b,u)$ where either $ab \in E(G_1)$ and $v=u$, or $uv \in E(G_2)$ and $a=b$. The \emph{direct product} $G_1 \times G_2$ consists of edges of the form $(a,v)(b,u)$ where $ab \in E(G_1)$ and $uv \in E(G_2)$. This product is also known as the \emph{tensor product}, the \emph{Kronecker product} and the \emph{cross product}. The \emph{strong product} $G_1 \boxtimes G_2$ is defined as $(G_1 \CartProd G_2) \cup (G_1 \times G_2)$. For a graph product $\bullet\in \{\CartProd, \times, \boxtimes\}$, graph classes $\G_1$ and $\G_2$, and graph $H$, let $\G_1 \bullet \G_2$ be the graph class $\{G_1 \bullet G_2: G_1 \in \G_1, G_2 \in \G_2\}$ and let $\G_1 \bullet H$ be the graph class $\{G_1 \bullet H: G_1 \in \G_1\}$. 

The following well-known properties of graph products are a straightforward consequence of their definition:
\begin{compactitem}
	\item for all $\bullet\in \{\CartProd, \times, \boxtimes\}$ and graphs $G_1$ and $G_2$ with subgraphs $H_1 \sse G_1$ and $H_2 \sse G_2$, we have $H_1 \bullet H_2 \sse G_1 \bullet G_2$;
	\item for all hereditary graph classes $\G_1$ and $\G_2$, we have $\G_1 \sse \G_1 \CartProd \G_2$ and $\G_2 \sse \G_1 \CartProd \G_2$;
	\item for all graphs $G_1$ and $G_2$ and vertices $v_i \in V(G_i)$ where $\deg_{G_i}(v_i)=d_i$ for $i \in \{1,2\}$:
	\begin{compactitem}
		\item $\deg_{G_1 \CartProd G_2}((v_1,v_2))=d_1+d_2$;
		\item $\deg_{G_1 \times G_2}((v_1,v_2))=d_1d_2$; and
		\item $\deg_{G_1 \boxtimes G_2}((v_1,v_2))=d_1+d_2+d_1d_2.$
	\end{compactitem}
\end{compactitem}

For a subgraph $Z \sse G_1 \bullet G_2$ of a graph product where $\bullet\in \{\CartProd, \times, \boxtimes\}$, the \emph{projection of $Z$ onto $G_1$} is the subgraph of $G_1$ induced by the set of vertices $u_1\in V(G_1)$ such that $(u_1,u_2)\in V(Z)$ for some $u_2\in V(G_2)$, and the \emph{projection of $Z$ onto $G_2$} is the subgraph of $G_2$ induced by the set of vertices $v_2\in V(G_2)$ such that $(v_1,v_2)\in V(Z)$ for some $v_1\in V(G_1)$.

\section{Complete Multipartite Subgraphs}\label{SectionMulti}	
For integers $d \geq 2$ and $n_1,\dots,n_d\geq 1$, a \emph{complete $d$-partite graph} $K_{n_1,\dots,n_d}$ has a partition of its vertex set, $(A_1,\dots,A_d)$, such that for all distinct $i,j \in [d]$ and $a_i \in A_i$ and $a_j \in A_j$, we have $a_ia_j \in E(K_{n_1,\dots,n_d})$ and $|A_i|=n_i$. When $d=2$, it is a \emph{complete bipartite graph}. Observe that the graph $K_{1,\dots,1}$ is a complete graph. In this section we characterise when a cartesian, direct and a strong product contains a given complete multipartite subgraph.

\subsection{Cartesian Product}
We begin with cartesian products. Let $G_1$ and $G_2$ be graphs. We say that $(u_1,v_1),\dots,(u_d,v_d) \in V(G_1 \CartProd G_2)$ are \emph{aligned} if $u_1=u_2=\ldots=u_d$ or $v_1=v_2=\ldots= v_d$. Observe that for a subgraph $H \sse G_1 \CartProd G_2$, if the vertex set $V(H)$ is aligned, then $H$ is isomorphic to a subgraph of $G_1$ or $G_2$. For a set $X \sse V(G_1 \CartProd G_2)$, if every pair of vertices in $X$ is aligned, then $X$ is aligned. Furthermore, if $x$ and $y$ are adjacent vertices in $G_1 \CartProd G_2$, then $x,y$ are aligned. Hence, we have the following.

\begin{lem}\label{CartProdNeighbourIntersection3}
	For all graphs $G_1$ and $G_2$, every clique in $G_1 \CartProd G_2$ is aligned.
\end{lem}

The next two lemmas are used in our characterisation of complete multipartite subgraphs in a cartesian product.

\begin{lem}\label{CartProdNeighbourIntersection1}
	For all graphs $G_1$ and $G_2$, if $X\sse V(G_1 \CartProd G_2)$ and $|\bigcap_{x \in X} N(x)| > 2$, then $X$ is aligned.
\end{lem}

\begin{proof}
	Let $X:=\{(u_1,v_1),\dots,(u_k,v_k)\}$ for some integer $k\geq 1$. If $k=1$ then the claim holds trivially. So assume that $k\geq 2$. For the sake of contradiction, suppose there exists distinct $i,j \in [k]$ such that $u_i\neq u_j$ and $v_i\neq v_j$. Then 
	\begin{equation*}
		\begin{split}	
			N((u_i,v_i)) \cap N((u_j,v_j))=&(\{(u_i,v): vv_i \in E(G_2)\} \cup \{(u,v_i): uu_i \in E(G_1)\}) \\
			&\cap (\{(u_j,v): vv_j \in E(G_2)\} \cup \{(u,v_j): uu_j \in E(G_1)\}) \\ 
			\sse& \{(u_i,v_j),(u_j,v_i)\}.
		\end{split}
	\end{equation*}
	But this contradicts the assumption that the intersection of the neighbourhoods of $(u_i,v_i)$ and $(u_j,v_j)$ has size greater than $2$. Thus, every pair of vertices in $X$ is aligned and hence, $X$ is aligned.
\end{proof}

\begin{lem}\label{CartProdNeighbourIntersection2}
	If $(u_1,v_1),(u_2,v_2)$ are distinct vertices that are aligned in $V(G_1 \CartProd G_2)$, then
	$N[(u_1,v_2)] \cap N[(u_1,v_2)]$ is aligned.
\end{lem}

\begin{proof}
	Without loss of generality, $v_1=v_2=v^\star$. Then
	\begin{equation*}
		\begin{split}	
			N((u_1,v^\star)) \cap N((u_2,v^\star))=& (\{(u_1,v): vv^\star \in E(G_2)\} \cup \{(u,v^\star): uu_1 \in E(G_1)\}) \\
			&\cap (\{(u_2,v): vv^\star \in E(G_2)\} \cup \{(u,v^\star): uu_2 \in E(G_1)\}) \\
			\sse& \{(x,v^\star): x \in V(G_1)\}.
		\end{split}
	\end{equation*}
	Hence $N[(u_1,v^\star)] \cap N[(u_2v^\star)]$ is aligned.
\end{proof}

We now characterise when a cartesian product contains a given complete multipartite subgraph. 

\begin{thm}\label{CompleteMultipartiteStrongProduct}
	For all integers $d \geq 2$ and $n_1,\dots,n_d \geq 1$, and for all graphs $G_1$ and $G_2$ with non-empty vertex sets and maximum degree $\Delta_1$ and $\Delta_2$ respectively, $K_{n_1,\dots,n_d}\subseteq G_1\CartProd G_2$ if and only if at least one of the following conditions hold:
	\begin{compactitem}
		\item $K_{n_1,\dots,n_d}$ is a subgraph of $G_1$ or $G_2$;
		\item $d=2$ and $(n_1,n_2)=(2,2)$ and $K_2 \sse G_1$ and $K_2 \sse G_2$; or
		\item $d=2$ and $(n_1,n_2)=(1,s)$ for some integer $s \geq 1$ and $\Delta_1 +\Delta_2 \geq s$.
	\end{compactitem}
\end{thm}	

\begin{proof}
	Clearly if $K_{n_1,\dots,n_d}$ is a subgraph of either $G_1$ or $G_2$, then it is a subgraph of $G_1 \CartProd G_2$. 
	
	Observe that $K_{1,s}$ is a subgraph of a graph $G$ if and only if the maximum degree of $G$ is at least $s$. As such, $K_{1,s}$ is a subgraph of $G_1 \CartProd G_2$ if and only if $\Delta_1+\Delta_2\geq s$. Now consider the bipartite graph $K_{2,2}$. Suppose $xy \in E(G_1)$ and $uv \in E(G_2)$. Let $A=\{(x,u),(y,v)\}$ and $B=\{(y,u),(x,v)\}$. Then $(A,B)$ defines a $K_{2,2}$ subgraph in $G_1 \CartProd G_2$. Now suppose that $K_{2,2} \sse G_1 \CartProd G_2$ and $K_2$ is not a subgraph of $G_1$. Then $E(G_1)=\emptyset$ and hence $V(G_1)$ is an independent set. As such, the connected components of $G_1 \CartProd G_2$ are isomorphic to the connected components of $G_2$. Since $K_{2,2}$ is connected, it is a subgraph of $G_2$, as required. 
	
	It remains to consider the cases when $(n_1,\dots,n_d)$ is not equal to $(1,s)$ or $(2,2)$. First, suppose $d \geq 3$ and that $K_{n_1,\dots,n_d} \sse G_1 \CartProd G_2$ with partition $(A_1,\dots,A_d)$. For every $i \in [d]$, let $(x_i,y_i) \in A_i$. Then $X:=\{(x_1,y_1),\dots,(x_d,y_d)\}$ induces a $K_d$ subgraph in $G_1 \CartProd G_2$. By \Cref{CartProdNeighbourIntersection3}, $X$ is aligned. Furthermore, for every $(u,v) \in V(K_{n_1,\dots,n_d})$, there exists distinct $i,j \in [d]$ such that $(u,v) \in N[(x_i,y_i)] \cap N[(x_j,y_j)]$. By \Cref{CartProdNeighbourIntersection2}, $V(K_{n_1,\dots,n_d})$ is aligned and hence $K_{n_1,\dots,n_d}$ is a subgraph of $G_1$ or $G_2$, as required.
	
	Now suppose $d=2$. Let $(A_1,A_2)$ be the bipartition of $K_{n_1,n_2}$ where $n_1=|A_1| \geq 2$ and $n_2=|A_2| \geq 3$. By \Cref{CartProdNeighbourIntersection1}, $A_1$ is aligned. By \Cref{CartProdNeighbourIntersection2}, $V(K_{n_1,n_2})$ is aligned and hence $K_{n_1,n_2}$ is a subgraph of $G_1$ or $G_2$.
\end{proof}

\subsection{Direct Product}	
The next theorem characterises when a direct product contains a given complete multipartite subgraph.
\begin{thm}\label{CompleteSubgraphsDirectProduct}
	For all integers $d \geq 2$ and $n_1,\dots,n_d \geq 1$, and for all graphs $G_1$ and $G_2$ with non-empty edge sets, $K_{n_1,\dots,n_d}\subseteq G_1\times G_2$ if and only if $K_{a_1,\dots,a_d} \subseteq G_1$ and $K_{b_1,\dots, b_d}\subseteq G_2$ where $a_i,b_i$ are positive integers and $n_i \leq a_ib_i$ for all $i \in [d]$.
\end{thm}	

\begin{proof}
	Suppose $K_{a_1,\dots,a_d} \subseteq G_1$ and $K_{b_1,\dots b_d}\subseteq G_2$. Let $(A_1,\dots A_d)$ be the partition of $K_{a_1,\dots,a_d}$ and $(B_1,\dots, B_d)$ be the partition of $K_{b_1,\dots b_d}$. Consider the graph $K_{a_1,\dots,a_d} \times K_{b_1,\dots b_d}$. For all distinct $i,j \in [d]$, every vertex $v_i \in (A_i,B_i)$ is adjacent to every vertex $v_j \in (A_j,B_j)$. Hence $((A_1,B_1),\dots,(A_d,B_d))$ defines a $K_{a_1b_1,\dots,a_db_d}$ subgraph in $K_{a_1,\dots,a_d} \times K_{b_1,\dots b_d}$. Since $n_i \leq a_ib_i$ for all $i \in [d]$, we have $K_{n_1,\dots,n_d}\sse G_1 \times G_2$.
	
	For the second direction, suppose $K_{n_1,\dots,n_d}\subseteq G_1\times G_2$. Let $(N_1,\dots,N_d)$ correspond to the partition of $K_{n_1,\dots,n_d}$. For $i \in \{1,2\}$ and $j \in [d]$, let $N^{(i)}_j$ be the vertices in $N_j$ projected onto $G_i$. By the definition of the direct product, $\{(x,v_2): x \in V(G_1)\}$ and $\{(v_1,y): y \in V(G_2)\}$ are independent sets for all $v_1 \in V(G_1)$ and $v_2 \in V(G_2)$. As such, $N_j^{(i)}$ and $N_k^{(i)}$ are disjoint for distinct $j,k \in [d]$ and $i \in \{1,2\}$. Moreover, as every vertex in $N_j$ is adjacent to every vertex in $N_k$, it follows that $K_{a_1\dots,a_d} \sse G_1[N_1^{(1)} \cup \dots N_d^{(1)}]$ and $K_{b_1,\dots, b_d} \sse G_2[N_1^{(2)} \cup \dots N_d^{(2)}]$ where $a_j=|N_j^{(1)}|$, and $b_j=|N_j^{(2)}|$ for all $j \in [d]$. Since there are $a_jb_j$ vertices in $(N_j^{(1)},N_j^{(2)})$, we have $n_j \leq a_jb_j$ for all $j \in [d]$, as required.
\end{proof}	

\subsection{Strong Product}
Let $K_{n_1,\dots,n_d,\overline{x}}$ be the graph obtained from the complete multipartite graph $K_{n_1,\dots,n_d,x}$ by adding an edge between each pair of vertices in the part of size $x\geq 0$. More formally, $V(K_{a_1,\dots,a_d,\overline{x}})=A_1\cup \dots \cup A_d \cup X$, such that $A_1,\dots,A_d,X$ are pairwise disjoint sets where for distinct $j,k \in [d]$, we have $|A_j|=a_j$, $|X|=x$, and $uv,vw_1,w_1w_2\in E(K_{a_1,\dots,a_d,\overline{x}})$ where $u\in A_j$, $v\in A_k$, and $w_1,w_2\in X$. Observe that $K_{n_1,\dots,n_d,\overline{x}}=K_{n_1,\dots,n_d,1,\dots,1}$ and $K_{n+c}=K_{{1,\dots,1},\bar{c}}$.

\begin{lem}
	\label{CompleteMultipartiteStrongProductLowerBound}
	If $K_{a_1,\dots,a_d,\overline{x}}\subseteq G_1$ and $K_{b_1,\dots,b_d,\overline{y}}\subseteq G_2$, then 
	$$K_{a_1b_1+a_1y+b_1x,\dots,a_db_d+a_dy+b_dx,\overline{xy}} \subseteq G_1\boxtimes G_2.$$ 
\end{lem}

\begin{proof}
	Let $A_1,\dots,A_d,X$ be the subsets of $V(G_1)$ defining a $K_{a_1,\dots,a_d,\overline{x}}$ subgraph of $G_1$. 	Let $B_1,\dots,B_d,Y$ be the subsets of $V(G_2)$ defining a $K_{b_1,\dots,b_d,\overline{y}}$ subgraph of $G_2$. 
	Then $(A_1 \times B_1)\cup(A_1\times Y)\cup(B_1\times X),\dots, (A_d \times B_d)\cup(A_d\times Y)\cup(B_d\times X), (X\times Y)$ define a $K_{a_1b_1+a_1y+b_1x,\dots,a_db_d+a_dy+b_dx,\overline{xy}}$ subgraph in $G_1\boxtimes G_2$. 	
\end{proof}

The next theorem characterises when a strong product contains a given complete multipartite subgraph.
\begin{thm}\label{StrongCMS}
	For all integers $d \geq 2$ and $n_1,\dots,n_d \geq 1$, and for all graphs $G_1$ and $G_2$, we have $K_{n_1,\dots,n_d} \subseteq G_1\boxtimes G_2$ if and only if there exists non-negative integers $a_1,\dots,a_d,b_1,\dots,b_d,z_1,\dots,z_d,x,y$ such that:
	\begin{compactitem}
		\item $K_{a_1,\dots,a_d,\overline{x}} \subseteq G_1$;
		\item $K_{b_1,\dots,b_d,\overline{y}} \sse G_2$;
		\item $n_j \leq a_jb_j+a_jy+b_jx+z_j$ for all $j \in [d]$; and 
		\item $z_1+\dots+ z_d \leq xy$.
	\end{compactitem}
\end{thm}

\begin{proof}
	First, suppose that $K_{n_1,\dots,n_d}$ is a subgraph in $G_1\boxtimes G_2$. Let $(N_1,\dots,N_d)$ be the partition of the $K_{n_1,\dots,n_d}$ subgraph. For all $i \in \{1,2\}$ and $j \in [d]$, let $N_j^{(i)}$ be the vertex set for the projection of $G_1\boxtimes G_2[N_j]$ onto $G_i$, let $\tilde{N}_j^{(i)}:=N_j^{(i)}\setminus (\bigcup_{h \in [d],h \neq j} N_h^{(i)})$ and let $Z^{(i)}:= \bigcup_{j \in [d]} (N_j^{(i)}\setminus \tilde{N}_j^{(i)})$. Then $(\tilde{N}_1^{(i)},\dots, \tilde{N}_d^{(i)},Z^{(i)})$ are the colour classes of a complete multipartite subgraph of $G_i$. For all $j \in [d]$, let $a_j:=|\tilde{N}_j^{(1)}|$, $b_j:=|\tilde{N}_j^{(2)}|$, $x:=|Z^{(1)}|$ and $y:=|Z^{(2)}|$. Thus $K_{a_1,\dots,a_d,x} \sse G_1$ and $K_{b_1,\dots,b_d,y}\subseteq G_2$. Consider distinct vertices $u_1,v_1 \in Z^{(1)}$. Then there exists distinct $j,k \in [d]$ such that $(u_1,u_2) \in N_j$ and $(v_1,v_2) \in N_k$ for some $u_2,v_2 \in V(G_2)$ which implies that $u_1v_1\in E(G_1)$. Similarly, the vertices in $Z^{(2)}$ are also pairwise adjacent in $G_2$. Hence, $K_{a_1,\dots,a_d,\overline{x}}\subseteq G_1$ and $K_{b_1,\dots,b_d,\overline{y}}\subseteq G_2$. Let $Z:=Z^{(1)} \times Z^{(2)}$ and $z_j:=|Z \cap N_j|$ for all $j \in [d]$. Then $z_1+\dots z_d \leq |Z|=xy$. Since $N_j \subseteq (\tilde{N}_j^{(1)} \times N_j^{(2)}) \cup (\tilde{N}_j^{(1)} \times Z_2) \cup (Z_1 \cap N_j^{(2)})$ we have $n_j\leq a_jb_j+a_jy+b_jx+z_j$ for all $j \in [d]$ as required.
	
	Now suppose that $K_{a_1,\dots,a_d,\overline{x}}\subseteq G_1$ and $K_{b_1,\dots,b_d,\overline{y}}$ where for all $j \in [d]$, $a_j,b_j,x,y$ are non-negative integers. By \cref{CompleteMultipartiteStrongProductLowerBound}, we have 
	$K_{n_1,\dots,n_d,\overline{xy}} \subseteq G_1\boxtimes G_2$. 
	Splitting the colour class of size $xy$ into sets of size $z_1,\dots, z_d$, and combining $z_j$ with the $j^{th}$ colour classes for all $j$, we obtain a
	$K_{a_1b_1+a_1y+b_1x+z_1,\dots, a_db_d+a_dy+b_dx+z_d}$ subgraph in $G_1\boxtimes G_2$. Since $n_j \leq a_jb_j+a_jy+b_jx+z_j$ for all $j \in [d]$, we have
	$K_{n_1,\dots,n_d} \subseteq G_1\boxtimes G_2$.
\end{proof}
For a graph $G$, let $\omega(G)$ be the maximum size of a clique in $G$. To illustrate the application of \Cref{StrongCMS}, we show how it implies the following well-known result.
\begin{cor}
	For all graphs $G$ and $H$, $\omega(G\boxtimes H)=\omega(G)\omega(H)$.
\end{cor}
\begin{proof}
	Let $c:=\omega(G)$ and $d:=\omega(H)$. Since $K_{0,\dots,0,\bar{c}}\sse G$ and $K_{0,\dots,0,\bar{d}}\sse H$, by \Cref{StrongCMS} we have $K_{1,\dots,1}=K_{cd}\sse G \boxtimes H$ by setting $z_i:=1$ for all $i \in [cd]$. Hence, $\omega(G\boxtimes H)\geq \omega(G)\omega(H)$. Now suppose that $K_n=K_{1,\dots, 1}\sse G \boxtimes H$. By \Cref{StrongCMS}, there exists non-negative integers $a_1,\dots,a_n,b_1,\dots,b_n,z_1,\dots,z_n,x,y$ such that $K_{a_1,\dots,a_n,\overline{x}} \subseteq G_1$; and $K_{b_1,\dots,b_n,\overline{y}} \sse G_2$; and:
	\begin{align}
		&n_i \leq a_ib_i+a_iy+b_ix+z_i \text{ for all $i \in [n]$}; \text{ and}\label{eqnA}\\
		&z_1+\dots+ z_n \leq xy\label{eqnB}.
	\end{align}
	Let $A:=\{i\in [n]:a_i\geq 1\}$, $B:=\{i\in [n]:b_i\geq 1\}$ and $Z:=\{i \in [n]:z_i\geq 1\}$. Then $|A|+|B|+|Z|\geq n$ since by \Cref{eqnA}, $a_i$, $b_i$ or $z_i$ is at least $1$ for all $i \in [n]$. By \Cref{eqnB}, $|Z|\leq xy$. Now $K_{{1,\dots, 1},\bar{x}}=K_c \sse G$, and $K_{{1,\dots, 1},\bar{y}}=K_d\sse H$ where $c:=|A|+x$ and $d:=|B|+y$. Therefore, $cd=(|A|+x)(|B|+y)\geq |A|+|B|+xy\geq n$. Hence $\omega(G)\omega(H)\geq \omega(G\boxtimes H)$, as required.
\end{proof}

The case of complete bipartite subgraphs is of particular interest. The following result generalises a lemma due to \citet[Lemma 7.2]{BGKTW2021twinsmall}.

\begin{cor}\label{StrongCompleteBipartite}
	For all integers $t \geq s \geq 1$ and $\Delta \geq 1$, for all graphs $G$ and $H$ where $G$ is $K_{s,t}$-free and $H$ has maximum degree $\Delta$, $G \boxtimes H$ is $K_{(s-1)(\Delta+1)+1,(s+t)(\Delta+1)}$-free. Moreover, for every integer $n\geq 1$, there exists a $K_{s,t}$-free graph $\tilde{G}$ and a graph $\tilde{H}$ with maximum degree $\Delta$ such that $K_{(s-1)(\Delta+1),n} \sse \tilde{G} \boxtimes \tilde{H}$.
\end{cor}
\begin{proof}
	For the sake of contradiction, suppose that $K_{(s-1)(\Delta+1)+1,(s+t)(\Delta+1)}\sse G \boxtimes H$. By \Cref{StrongCMS}, there exists non-negative integers $a_1,a_2,b_1,b_2,z_1,z_2,x,y$ such that $K_{a_1,a_2,\overline{x}} \subseteq G_1$; and $K_{b_1,b_2,\overline{y}} \sse G_2$; and
	\begin{align}
		(s-1)(\Delta+1)+1 &\leq a_1(b_1+y)+b_1x+z_1; \label{eqn1}\\
		(s+t)(\Delta+1) &\leq a_2(b_2+y)+b_2x+z_2; \text{ and} \label{eqn2}\\
		z_1+z_2 &\leq xy. \label{eqn3}
	\end{align}

	Observe that if $b_1$ and $y$ are both equal to $0$ then \Cref{eqn1,eqn3} cannot both be satisfied. As such, $b_2+y\leq \Delta+1$ since $H$ has maximum degree $\Delta$. Similarly, $b_2$ and $y$ cannot both equal to $0$ as this will violate \Cref{eqn2}. Thus, $b_1+y\leq \Delta+1$. By \Cref{eqn3}, $z_1\leq xy$ and $z_2\leq xy$. By \Cref{eqn1}, we have $(\Delta+1)(s-1)+1\leq a_1(b_1+y)+b_1x+xy\leq (a_1+x)(\Delta+1)$. As such, $a_1+x>s-1$. Similarly, by \Cref{eqn2}, $(s+t)(\Delta+1) \leq a_2(b_2+y)+b_2x+xy\leq (a_2+x)(\Delta+1)$. As such, $a_2+x\geq s+t$. This forces $K_{s,t}\sse G$, a contradiction.	
	
	Finally, let $\tilde{G}:=K_{s-2,n,1}$ and $\tilde{H}:=K_{\Delta,0,1}$. Then $\tilde{G}$ is $K_{s,t}$-free and $H$ has maximum degree $\Delta$. By \Cref{StrongCMS}, $K_{n_1,n_2}\sse \tilde{G}\boxtimes \tilde{H}$ where $n_1=\Delta(s-2)+(s-2)+\Delta +1=(s-1)(\Delta+1)$ and $n_2=n$.
\end{proof}

\section{Degeneracy}\label{SectionDegen}
Recall that the \emph{degeneracy} of a graph $G$, $\degen(G)$, is the minimum integer $d$ such that every subgraph of $G$ has minimum degree at most $d$. Equivalently, the \emph{degeneracy} of $G$ is the maximum, taken over all subgraphs $H$ of $G$, of the minimum degree of $H$. 

\citet{bickle2010cores} determined the degeneracy of a cartesian product. In particular, for all graphs $G_1$ and $G_2$, $\degen(G_1\CartProd G_2)=\degen(G_1)+\degen(G_2).$

%

\subsection{Direct Product}
We now prove tight bounds for the degeneracy of a direct product. We make use of the following observation.

\begin{obs}\label{DegenDirectCBG}
	For all integers $t_i\geq s_i \geq 1$ where $i\in \{1,2\}$, 
	$$\degen(K_{s_1,t_1}\times K_{s_2,t_2})=\min{\{s_1t_2,s_2t_1\}}.$$
\end{obs}
This follows from the fact that $K_{s_1,t_1}\times K_{s_2,t_2}$ is the disjoint union of $K_{s_1t_2,s_2t_1}$ and $K_{s_1s_2,t_1t_2}$, and that $\degen(K_{s,t})=\min\{s,t\}$.

\begin{thm}
	\label{DegenDirect}
	For $i\in\{1,2\}$, let $G_i$ be a graph with maximum degree $\Delta_i$ and $\degen(G_i)=d_i$ that contains $K_{s_i,t_i}$ as a subgraph where $s_i \leq t_i$. Then:
	$$\max{\{d_1d_2,\min{\{s_1t_2,s_2t_1\}},\min{\{\Delta_1,\Delta_2\}}\}} \leq \degen(G_1 \times G_2) \leq \min{\{d_1\Delta_2,d_2\Delta_1\}}.$$	
\end{thm}

\begin{proof}
	We first prove the lower bound. Let $Q_i$ be a subgraph of $G_i$ with $\delta(Q_i)=d_i$. Then $\delta(Q_1 \times Q_2)=d_1d_2$ and thus, $\degen(G_1 \times G_2) \geq d_1d_2$. Furthermore, since $G_i$ contains $K_{1,\Delta_i}$ and $K_{s_i,t_i}$ as a subgraph, by \Cref{DegenDirectCBG}, 
	$\degen(G_1 \times G_2)\geq \min\{s_1t_2,s_2t_1\}$ and $\degen(G_1 \times G_2) \geq \min{\{\Delta_1,\Delta_2\}}$.
	
	Now we prove the upper bound. Without loss of generality, $d_1\Delta_2 \leq d_2\Delta_1$. Let $Z$ be a subgraph of $G_1 \times G_2$. Our goal is to show that $\delta(Z)\leq d_1\Delta_2$. Let $X$ be the projection of $Z$ onto $G_1$. Since $G_1$ is $d_1$-degenerate, there exists a vertex $v_1 \in V(X)$ with $\deg_X(v_1)\leq d_1$. By construction of $X$, we have $(v_1,v_2) \in V(Z)$ for some $v_2 \in V(G_2)$. The neighbourhood of $(v_1,v_2)$ in $Z$ is a subset of $\{(u_1,u_2):u_1v_1 \in E(X), u_2v_2 \in E(G_2)\}$. Now $|\{(u_1,u_2):u_1v_1 \in E(X), u_2v_2 \in E(G_2)\}| \leq |\{u_1 \in V(X):u_1v_1 \in E(X)\}| |\{u_2 \in V(X):u_2v_2 \in E(G_2)\}|\leq d_1 \Delta_2$. Thus $(v_1,v_2)$ has degree at most $d_1\Delta_2$ in $Z$ and hence $\delta(Z)\leq d_1\Delta_2$ as required.
\end{proof}

When both graphs are regular, the upper and lower bounds in \Cref{DegenDirect} are equal. Furthermore, by \Cref{DegenDirectCBG}, the direct product of two complete bipartite graphs realises the upper bound in \Cref{DegenDirect}. We now show that there is a family of graphs that realises the lower bound in \Cref{DegenDirect}.

\begin{lem}\label{DegenDirectLBTight}
	For all integers $d_i,\Delta_i,s_i,t_i\geq 1$ where $s_i\leq d_i \leq \Delta_i$ and $s_i\leq t_i\leq \Delta_i$ for $i \in \{1,2\}$, there exists graphs $G_1$ and $G_2$ where for $i \in \{1,2\}$, $G_i$ has maximum degree $\Delta_i$, $\degen(G_i)=d_i$, and $G_i$ contains $K_{s_i,t_i}$ as a subgraph, such that
	$$\degen(G_1 \times G_2)= \max{\{d_1d_2,\min{\{s_1t_2,s_2t_1\}},\min{\{\Delta_1,\Delta_2\}}\}}.$$
\end{lem}

\begin{proof}
	Let $\tilde{d}:=\max{\{d_1d_2,\min{\{s_1t_2,s_2t_1\}},\min{\{\Delta_1,\Delta_2\}}\}}$.
	For $i \in \{1,2\}$, let $H_i$ be a $d_i$-regular graph and let $G_i$ be the disjoint union of $H_i$, $K_{s_i,t_i}$ and $K_{1,\Delta_i}$. Then $G_i$ has maximum degree $\Delta_i$, $\degen(G_i)=d_i$, and $G_i$ contains $K_{s_i,t_i}$ and $K_{1,\Delta_i}$ as subgraphs. 
	
	We now show that $\degen(G_1 \times G_2)=\tilde{d}.$ \Cref{DegenDirect} provides the lower bound. For the upper bound, our goal is to show that $\degen(J_1 \times J_2)\leq \tilde{d}$ whenever $J_i \in \{H_i,K_{s_i,t_i},K_{1,\Delta_i}\}$ for $i \in \{1,2\}$. This implies that $\degen(G_1 \times G_2)\leq \tilde{d}$ since every connected subgraph of $G_1 \times G_2$ is a subgraph of $J_1 \times J_2$ for some choice of $J_i \in \{H_i,K_{s_i,t_i},K_{1,\Delta_i}\}$ where $i \in \{1,2\}$. We proceed by case analysis.	 
	
	First, $\degen(H_1\times H_2)=d_1d_2$ since $H_1$ and $H_2$ are regular. By \Cref{DegenDirectCBG},
	\begin{align*}
		 \degen(K_{s_1,t_1}\times K_{s_2,t_2})&=\min{\{s_1t_2,s_2t_1\}},\\ 
		 \degen(K_{s_1,t_1}\times K_{1,\Delta_2})&=\min{\{t_1,\Delta_2\}},
		 \\ \degen(K_{1,\Delta_1}\times K_{s_2,t_2})&=\min{\{\Delta_1,t_2\}}, \text{ and}\\
		 \degen(K_{1,\Delta_1}\times K_{1,\Delta_2})&=\min{\{\Delta_1,\Delta_2\}}.
	\end{align*}
	Clearly the degeneracy is at most $\tilde{d}$ for each of the above graphs. Now consider the graph $H_1 \times K_{s_2,t_2}$. Each vertex has degree $d_1s_2$ or $d_1t_2$. Furthermore, the set of vertices with degree equal to $d_1t_2$ are an independent set. Hence, every subgraph of $H_1 \times K_{s_2,t_2}$ contains a vertex with degree at most $d_1s_2$. As such, $\degen(H_1 \times K_{s_2,t_2})=d_1s_2$. By the same reasoning: $\degen(H_1 \times K_{1,\Delta_2})=d_1$, $\degen(K_{s_1,t_1} \times H_2 )=s_1d_2$, and 		$\degen(K_{1,\Delta_1} \times H_2 )=d_2$. Again, the degeneracy is at most $\tilde{d}$ for each of the above graphs. Hence, $\degen(J_1 \times J_2)\leq \tilde{d}$ whenever $J_i \in \{H_i,K_{s_i,t_i},K_{1,\Delta_i}\}$ for $i \in \{1,2\}$, as required.
\end{proof}

We conclude that the upper and lower bounds in \Cref{DegenDirect} are tight.

\subsection{Strong Product}
We now consider the degeneracy of a strong product. As the next lemma illustrates, more cases arise for this graph product compared to the other two.
\begin{lem}\label{DegenStrongCBG}
	For all integers $t_i\geq s_i \geq 1$, $\degen(K_{s_1,t_1}\boxtimes K_{s_2,t_2})=f(s_1,t_1,s_2,t_2)$
	where $$f(s_1,t_1,s_2,t_2):=\max\{s_1+s_2+s_1s_2,\min{\{t_1+t_2,s_1(t_2+1),s_2(t_1+1)\}},\min{\{s_1t_2,s_2t_1\}}\}.$$
\end{lem}
\begin{proof}
	Let $\hat{d}:=f(s_1,t_1,s_2,t_2)$. Let $(S_1,T_1)$ and $(S_2,T_2)$ be the bipartition of $K_{s_1,t_1}$ and $K_{s_2,t_2}$ respectively. Let $G:=K_{s_1,t_1}\boxtimes K_{s_2,t_2}$. Let ${A}:=S_1 \times S_2$, ${B}:=S_1 \times T_2$, ${C}:=T_1 \times S_2$, and ${D}:=T_1 \times T_2$.
	
	We first prove the lower bound by showing that there exists a subgraph of $G$ with minimum degree $\hat{d}$. Observe that $\delta(G)=s_1+s_2+s_1s_2$. By \Cref{DegenDirectCBG}, there exists a subgraph $J$ of $K_{s_1,t_1}\times K_{s_2,t_2}\sse G$ with $\delta(J)=\min{\{s_1t_2,s_2t_1\}}$. Let $H$ be the subgraph of $G$ induced by ${A}\cup {B} \cup {C}$. For this subgraph, $\deg_H({a})=t_1+t_2$, $\deg_H({b})=s_2(t_1+1)$, and $\deg_H({c})=s_1(t_2+1)$ for every ${a} \in {A}$, ${b} \in {B}$ and ${c} \in {C}$. As such, $\delta(H)=\min{\{t_1+t_2,s_1(t_2+1),s_2(t_1+1)\}}$. Therefore, either $G$, $J$ or $H$ has minimum degree $\hat{d}$.
	
	We now prove the upper bound. Let $Z$ be a subgraph of $G$. We proceed by case analysis to show that the minimum degree of $Z$ is at most $\hat{d}$. 
	
	First, if $V(Z)$ is a subset of either ${A}$, ${B}$, ${C}$ or ${D}$, then $\delta(Z)=0$ since these are independent sets. So we may assume that $V(Z)$ intersects at least two of those sets. Now if there exists a vertex ${v}\in V(Z)\cap {D}$, then $\deg_Z({v})\leq s_1+s_2+s_1s_2 \leq \hat{d}$. So we may assume that $V(Z)\cap {D}=\emptyset$. Now if $V(Z)\cap {A}$, $V(Z)\cap {B}$, and $V(Z)\cap {C}$ are all non-empty, then $\delta(Z)\leq\min{\{t_1+t_2,s_1(t_2+1),s_2(t_1+1)\}}$. Furthermore, if $V(Z)\sse {B}\cup {C}$, then $\delta(Z)\leq\min{\{s_1t_2,s_2t_1\}}$. Finally, if $V(Z)\sse {A}\cup {B}$ or $V(Z)\sse {A}\cup {C}$, then $\delta(Z)\leq \max\{s_1,s_2\}$. In each case, $\delta(Z)\leq \hat{d}$ as required. 
\end{proof}

The next theorem provides tight bounds for the degeneracy of a strong product.
\begin{thm}
	\label{DegenStrong}
	For $i\in\{1,2\}$, let $G_i$ be a graph with maximum degree $\Delta_i$ and $\degen(G_i)=d_i$ that contains $K_{s_i,t_i}$ as a subgraph where $s_i \leq t_i$. Then
	$$g(d_1,\Delta_1,s_1,t_1,d_2,\Delta_2,s_2,t_2)\leq \degen(G_1 \boxtimes G_2) \leq h(d_1,\Delta_1,d_2,\Delta_2)$$
	where $f$ is specified by \Cref{DegenStrongCBG} and
	\begin{align*}
		g(d_1,\Delta_1,s_1,t_1,d_2,\Delta_2,s_2,t_2)&=\max{\{d_1+d_2+d_1d_2,f(s_1,t_1,s_2,t_2),\min{\{\Delta_1,\Delta_2\}}+1\}},\\
		h(d_1,\Delta_1,d_2,\Delta_2)&=d_1+d_2+\min{\{d_1\Delta_2,d_2\Delta_1\}}.
	\end{align*}
\end{thm}

\begin{proof}
	We first prove the lower bound. Let $Q_i$ be a subgraph of $G_i$ with $\delta(Q_i)=d_i$. Then $\delta(Q_1 \boxtimes Q_2)=d_1+d_2+d_1d_2$ and thus, $\degen(G_1 \boxtimes G_2)\geq d_1+d_2+d_1d_2$. Furthermore, since $G_i$ contains $K_{1,\Delta_i}$ and $K_{s_i,t_i}$ as a subgraph, by \Cref{DegenStrongCBG}, 
	$\degen(G_1 \boxtimes G_2)\geq f(s_1,t_1,s_2,t_2)$ and $\degen(G_1 \boxtimes G_2)\geq\min{\{\Delta_1,\Delta_2\}}+1$. 
	
	We now prove the upper bound. Let $Z$ be a subgraph of $G_1 \boxtimes G_2$. Our goal is to show that $\delta(Z)\leq h(d_1,\Delta_1,d_2,\Delta_2)$. Without loss of generality, $d_1\Delta_2 \leq d_2\Delta_1$. Let $Z$ be a subgraph of $G_1 \boxtimes G_2$ and let $X$ be the projection of $Z$ onto $G_1$. Since $G_1$ is $d_1$-degenerate, there exists a vertex $v_1 \in V(X)$ with $\deg_X(v_1)\leq d_1$. Let $Y$ be the subgraph of $G_2$ induced by the set of vertices $v_2$ in $G_2$ such that $(v_1,y)\in V(Z)$. Since $G_2$ is $d_2$-degenerate, there exists a vertex $v_2 \in V(Y)$ with $\deg_Y(v_2)\leq d_2$.	By construction of $Y$, $(v_1,v_2) \in V(Z)$. By the definition of the strong product, 
		$$\deg_Z((v_1,v_2)) \leq |N_X(v_1)|+|N_Y(v_2)|+|N_X(v_1)||N_{G_2}(v_2)|\leq d_1+d_2+d_1\Delta_2.$$
	Hence $\delta(Z)\leq h(d_1,\Delta_1,d_2,\Delta_2)$ as required.
\end{proof}

We now construct a family of graphs that realises the lower bound in \Cref{DegenStrong}.

\begin{lem}
	For all integers $d_i,\Delta_i,s_i,t_i\geq 1$ where $s_i\leq d_i \leq \Delta_i$ and $s_i\leq t_i\leq \Delta_i$ for $i \in \{1,2\}$, there exists graphs $G_1$ and $G_2$ where for $i \in \{1,2\}$, $G_i$ has maximum degree $\Delta_i$, $\degen(G_i)=d_i$, and $G_i$ contains $K_{s_i,t_i}$ as a subgraph, such that
	$$\degen(G_1 \boxtimes G_2)=g(d_1,\Delta_1,s_1,t_1,d_2,\Delta_2,s_2,t_2)$$
	
	where $g$ is specified in \Cref{DegenStrong}.
\end{lem}

\begin{proof}
	Let $\tilde{d}:=g(d_1,\Delta_1,s_1,t_1,d_2,\Delta_2,s_2,t_2)$.
	Our proof parallels the proof of \Cref{DegenDirectLBTight}. For $i \in \{1,2\}$, let $H_i$ be a $d_i$-regular graph and let $G_i$ be the disjoint union of $H_i$, $K_{s_i,t_i}$ and $K_{1,\Delta_i}$. Then $G_i$ has maximum degree $\Delta_i$, $\degen(G_i)=d_i$, and $G_i$ contains $K_{s_i,t_i}$ and $K_{1,\Delta_i}$ as subgraphs. 
	
	We now show that $\degen(G_1 \boxtimes G_2)=\tilde{d}.$ \Cref{DegenStrong} provides the lower bound. For the upper bound, our goal is to show that $\degen(J_1 \boxtimes J_2)\leq \tilde{d}$ whenever $J_i \in \{H_i,K_{s_i,t_i},K_{1,\Delta_i}\}$ for $i \in \{1,2\}$. We proceed by case analysis.	 
	
	First, $\degen(H_1\boxtimes H_2)=d_1+d_2+d_1d_2$ since $H_1$ and $H_2$ are regular. By \Cref{DegenStrongCBG},
	\begin{align*}
		\degen(K_{s_1,t_1}\boxtimes K_{s_2,t_2})&=f(s_1,t_1,s_2,t_2) \text{ where $f$ is specifed by \Cref{DegenStrongCBG}},\\ 
		\degen(K_{s_1,t_1}\boxtimes K_{1,\Delta_2})&=\max{\{2s_1+1,\min{\{t_1,s_1\Delta_2\}}\}},
		\\ 
		\degen(K_{1,\Delta_1}\boxtimes K_{s_2,t_2})&=\max{\{2s_2+1,\min{\{s_2\Delta_1,t_2\}}\}}\text{, and}\\
		\degen(K_{1,\Delta_1}\boxtimes K_{1,\Delta_2})&=\min{\{\Delta_1,\Delta_2\}}+1.
	\end{align*}
	For each of the above graphs, the degeneracy is at most $\tilde{d}$. Now consider the graph $H_1 \boxtimes K_{s_2,t_2}$. Each vertex in $H_1 \boxtimes K_{s_2,t_2}$ has degree $d_1+s_2+d_1s_2$ or $d_1+t_2+d_1t_2$. Let $A:=\{v \in V(H_1 \boxtimes K_{s_2,t_2}):\deg(v)=d_1+s_2+d_1s_2\}$ and $B:=\{v \in V(H_1 \boxtimes K_{s_2,t_2}):\deg(v)=d_1+t_2+d_1t_2\}$. Let $Z$ be a subgraph of $H_1 \boxtimes K_{s_2,t_2}$. If $V(Z)\cap A$ is non-empty, then $\delta(Z)\leq d_1+s_2+d_1s_2$. Otherwise, $V(Z)\sse B$ in which case $\delta(Z)\leq d_1$. Thus, $\degen(H_1 \boxtimes K_{s_2,t_2})=d_1+s_2+d_1s_2$. By the same reasoning: $\degen(H_1 \boxtimes K_{1,\Delta_2})=2d_1+1$, $\degen(K_{s_1,t_1} \boxtimes H_2 )=d_2+s_1+d_2s_1$, and $\degen(K_{1,\Delta_1} \boxtimes H_2 )=2d_2+1$. Again, the degeneracy is at most $\tilde{d}$ for each of the above graphs. Having considered all possibilities, it follows that $\degen(J_1 \boxtimes J_2)\leq \tilde{d}$ whenever $J_i \in \{H_i,K_{s_i,t_i},K_{1,\Delta_i}\}$ for $i \in \{1,2\}$, as required.
\end{proof}

The next lemma describes a family of graphs that realises the upper bound in \Cref{DegenStrong}.

\begin{lem}
	For all integers $k_1,k_2,d_1,d_2\geq 1$ there exists graphs $G_1$ and $G_2$ where for $i \in \{1,2\}$, $G_i$ has maximum degree $\Delta_i:=k_id_i$, and $\degen(G_i)=d_i$, such that
	$$\degen(G_1 \boxtimes G_2)= d_1+d_2+\min{\{d_1\Delta_2,d_2\Delta_1\}}.$$
\end{lem}

\begin{proof}
	For $i \in \{1,2\}$, let $G_i$ be a graph with vertex partition $(A_i,B_i)$ where $A_i$ is a clique on $d_i+1$ vertices and $B_i$ is an independent set of size $(d_i+1)(k_i-1)$ such that each $a_i \in A_i$ has $d_i(k_i-1)$ neighbours in $B_i$ and each $b_i \in B_i$ has $d_i$ neighbours in $A_i$. Then $G_i$ has maximum degree $\Delta_i$ and $\degen(G_i)=d_i$.
	
	We now show that $\degen(G_1 \boxtimes G_2)=d_1+d_2+\min{\{d_1\Delta_2,d_2\Delta_1\}}$. \Cref{DegenStrong} provides the upper bound. Let ${X}:=A_1 \times A_2$, ${Y}:=B_1 \times A_2$, and ${Z}:=A_1 \times B_2$. Let $H$ be the subgraph of $G_1 \boxtimes G_2$ induced on ${X} \cup {Y} \cup {Z}$. Let ${x}\in {X}$, ${y}\in {Y}$ and ${z}\in {Z}$. Then
	\begin{align*}
		\deg_H({x})&=|N_H({x})\cap {X}|+|N_H({x})\cap {Y}|+|N_H({x})\cap {Z}|\\ 
		&=\big((d_1+1)(d_2+1)-1\big)+\big(d_2(\Delta_1-d_1)\big)+\big(d_1(\Delta_2-d_2)\big),
		\\ 
		\deg_H({y})&=|N_H({y})\cap {X}|+|N_H({y})\cap {Y}|+|N_H({y})\cap {Z}|\\
		&=\big(d_1(d_2+1)\big)+\big(d_2\big)+\big(d_1(\Delta_2-d_2)\big)= d_1+d_2+\Delta_2d_1, \text{ and}
		\\ 
		\deg_H({z})&=|N_H({z})\cap {X}|+|N_H({z})\cap {Y}|+|N_H({z})\cap {Z}|\\
		&=\big(d_2(d_1+1)\big)+\big(d_2(\Delta_1-d_1)\big)+\big(d_1\big)=d_1+d_2+d_2\Delta_1.
	\end{align*}
	In which case, the minimum degree of $H$ is	$d_1+d_2+\min{\{d_1\Delta_2,d_2\Delta_1\}}$ as required.
\end{proof}
We conclude that the upper and lower bounds in \Cref{DegenStrong} are tight.

\section{Pathwidth and Treewidth}\label{Sectionpwtw}
We now consider the treewidth and pathwidth of a graph product. For graphs $G$ and $H$, an \emph{$H$-decomposition} of $G$ consists of a collection $\B=\{B_x\sse V(G): x \in V(H)\}$ of \emph{bags} with the following properties. 

\begin{compactitem}
	\item For each $v \in V(G)$, $H[\{x \in V(H): v \in B_x\}]$ is a non-empty connected subgraph of $H$; and
	\item For each $uv \in E(G)$, there is a node $x \in V(H)$ such that $u,v \in B_x$.
\end{compactitem}

The \emph{width} of an $H$-decomposition is $\max\{|B_x|-1:x \in V(H)\}$. We denote an $H$-decomposition by the pair $(H,\B)$. A \emph{tree-decomposition} is a $T$-decomposition for some tree $T$. The \emph{treewidth} of a graph $G$, $\tw(G)$, is the minimum width of a tree-decomposition of $G$. A \emph{path-decomposition} is a $P$-decomposition for some path $P$. The \emph{pathwidth} of a graph $G$, $\pw(G)$, is the minimum width of a path-decomposition of $G$. Since a path is a tree, $\tw(G)\leq \pw(G)$. Treewidth and pathwidth respectively measure how similar a graph is to a tree and to a path. These two graph parameters are minor-closed. Tree-decompositions and path-decompositions were introduced by Robertson and Seymour \cite{robertson1986algorithmic,robertson1983graph}; the more general notion of $H$-decompositions were introduced by \citet{diestel2005graph}.

Let $G$ be a graph. We say that $X,Y \sse V(G)$ \textit{touch} if $X \cap Y \neq \emptyset$ or there is an edge of $G$ between $X$ and $Y$. A \textit{bramble}, $\mathcal{B}$, is a set of pairwise touching connected subgraphs. A set $S \sse V(G)$ is a \textit{hitting set} of $\mathcal{B}$ if $S$ intersects every element of $\mathcal{B}$. The \textit{order} of $\mathcal{B}$ is the minimum size of a hitting set of $\mathcal{B}$. The canonical example of a bramble of order $\ell$ is the set of crosses (union of a row and column) in the $\ell \times \ell$ grid. The following `Treewidth Duality Theorem' shows the intimate relationship between treewidth and brambles.
\begin{thm}[\cite{seymour1993graph}]\label{BrambleDuality}
	A graph $G$ has treewidth at least $\ell$ if and only if $G$ contains a bramble of order at least $\ell+1$.
\end{thm}

Many papers have studied the treewidth and pathwidth for various families of graphs that are defined by a graph product, see \cite{CK2006treewidth, chvatalova1974optimal, djelloul2009logical, fitzgerald1974indexing, harper1966optimal, harper1999isoperimetric, harper2003bandwidth, KA2002bandwidth, KOY2014product, LTS1981bandwidth, moghadam2005bandwidth, OS2011bandwidth}. In this section, we continue this work by presenting new lower bounds for cartesian and strong product as well as characterising when the cartesian, direct and strong products have bounded treewidth and when they have bounded pathwidth.

\subsection{Cartesian and Strong Product}\label{SectiontwCart}
First consider the treewidth of the cartesian and strong products. The following upper bound is well-known (see \cite{BGHK1995approximating} for an implicit proof).

\begin{lem}\label{twStrongProduct}
	For all graphs $G_1$ and $G_2$,
	\begin{align*}
		\tw(G_1 \CartProd G_2)\leq \tw(G_1\boxtimes G_2)&\leq (\tw(G_1)+1)\textsf{v}(G_2)-1, \text{ and}\\
		\pw(G_1 \CartProd G_2)\leq\pw(G_1\boxtimes G_2)&\leq (\pw(G_1)+1)\textsf{v}(G_2)-1.
	\end{align*}
\end{lem}
The proof of \Cref{twStrongProduct} in \cite{BGHK1995approximating} shows the following, more general result which we use in \Cref{SectionDirectProduct}.

\begin{lem}\label{HDecompStrongProduct}
	For all graphs $G_1$, $G_2$ and $H$, if $G_1$ has an $H$-decomposition with width at most $k$, then $G_1 \boxtimes G_2$ has an $H$-decomposition with width at most $(k+1)\textsf{v}(G_2)-1$.
\end{lem}
\begin{proof}[Proof Sketch]
	 Let $(H,\W)$ be an $H$-decomposition of $G_1$ with width $k$. Modify $(H,\W)$ to obtain a tree-decomposition $(H,\B)$ of $G_1 \boxtimes G_2$ by setting $B_t:=\{(v,u):v \in W_t, u \in V(G_2)\}$ for all $t \in V(H)$. Then $(H,\B)$ is an $H$-decomposition of $G_1 \boxtimes G_2$ with width at most $(k+1)\textsf{v}(G_2)-1$.
\end{proof}

A natural question is whether the upper bound in \Cref{twStrongProduct} is tight up to a constant factor. The following result shows that this is not the case.

\begin{prop}\label{NaiveTight}
	For all $n\geq k+1\geq 0$, there exists a connected graph $G_{k,n}$ such that $\tw(G_{k,n})= k$ and $\textsf{v}(G_{k,n})=n$ and
		$$\tw(G_{k,n}\boxtimes G_{k,n})=\Theta(n+k^2).$$
\end{prop}

\begin{proof}
	We make no attempt to optimise the constants in this proof. Let $G_{k,n}$ be the graph that consists of a path $P_{\tilde{n}}=(v_0,\dots,v_{\tilde{n}-1})$ on $\tilde{n}=n-k$ vertices and a complete graph $K_{k+1}$ where $V(P_n)\cap V(K_k)=\{v_0\}$. Then $\tw(G_{k,n})= k$ and $\textsf{v}(G_{k,n})=n$.
	
	We now show that $\tw(G_{k,n}\boxtimes G_{k,n})=\Theta(n+k^2)$.	 For the lower bound, since $\G_{k,n}$ is connected, by \Cref{twCartProdLower}, $\tw(G_{k,n}\boxtimes G_{k,n})\geq n-1$. Moreover, since $K_k\boxtimes K_k\simeq K_{k^2}$, $\tw(G_{k,n}\boxtimes G_{k,n})\geq \tw(K_{k^2})=k^2-1$. Thus $\tw(G_{k,n}\boxtimes G_{k,n})=\Omega(n+k^2)$.
	
	For the upper bound, we construct a tree decomposition of $G_{k,n}\boxtimes G_{k,n}$ with width $O(n+k^2)$. We begin by specifying the bags. For $i,j,\ell \in [\tilde{n}-1]$, let
	\begin{align*}
		X&:=\{(v_0,u),(u,v_0):u \in V(G_{k,n})\},\\
		W_y&:=\{(a,b):a,b \in K_k\}\cup X,\\
		C_{u_i}&:=\{(a,v_{j-1}),(a,v_j): a \in K_k\}\cup X,\\
		D_{w_j}&:=\{(v_{i-1},a),(v_i,a): a \in K_k\}\cup X, \text{ and}\\
		L_{z_\ell}&:=\{(v_{\ell-1},v_s),(v_{\ell},v_s),(v_s,v_{\ell-1}),(v_s,v_{\ell}):s \in [\ell,n-1]\}\cup X.
	\end{align*}
	 Observe that $|X|\leq 2n$, $|W_y|\leq 2n+k^2$, $|C_i|=|D_i|\leq 2n+2k$ and $|L_{w_\ell}|\leq 6n$. Moreover, observe that $V(K_k\boxtimes K_k)\sse W_y$, $ 	V(K_k\boxtimes P_n)\sse \bigcup_{i \in [n-1]} C_{u_i}$, $V(P_n \boxtimes K_k)\sse \bigcup_{i \in [\tilde{n}-1]} D_{z_i}$, and $V(P_n \boxtimes P_n)\sse \bigcup_{\ell \in [\tilde{n}-1]} L_{w_\ell}$. Thus, every bag has size $O(n+k^2)$ and every vertex is in a bag. For the tree to index the decomposition, let $P^{(C)}:=(u_1,\dots,u_{\tilde{n}-1})$, $P^{(D)}:=(w_1,\dots,w_{\tilde{n}-1})$, and $P^{(L)}:=(z_1,\dots,z_{\tilde{n}-1})$ be paths on $\tilde{n}-1$ vertices. Let $T$ be the tree obtained by taking the disjoint union of $P^{(C)}$, $P^{(D)}$ and $P^{(L)}$, then adding the vertex $y$ and the edges $yu_1,yw_1,yz_1$ (see \Cref{fig:Gkn}). Let $\mathcal{W}:=\{W_t:t \in V(T)\}$. We claim that $(T,\mathcal{W})$ is a tree decomposition of $G_{k,n}\boxtimes G_{k,n}$.

	 \begin{figure}[!htb]
		\begin{center}
			\includegraphics[width=0.8\linewidth]{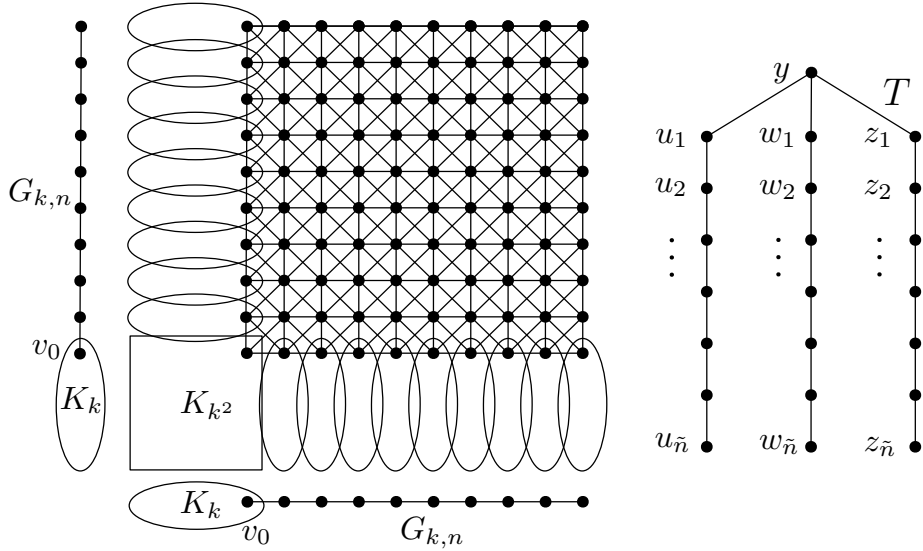}
			\caption{$G_{k,n}\boxtimes G_{k,n}$ with the tree $T$.}
			\label{fig:Gkn}
		\end{center}
	\end{figure}
	
	As noted earlier, every vertex is in a bag. Let $(a_1,a_2)(b_1,b_2)\in E(G_{k,n}\boxtimes G_{k,n})$. Suppose that $a_1$ or $a_2$ is equal to $v_0$. Then $(a_1,a_2)\in X$. Thus, since $(b_1,b_2)\in W_{t'}$ for some $t' \in V(T)$ and $X\sse W_{t}$ for all $t \in V(T)$, we have $(a_1,a_2)(b_1,b_2)\in W_{t'}$. Now suppose that $a_1,a_2\in K_{k+1}\bs \{v_0\}$. Then $N((a_1,a_2))\sse K_{k+1} \boxtimes K_{k+1}$ in which case $(a_1,a_2),(b_1,b_2)\in W_y$. Now if $a_1 \in K_{k+1}\bs \{v_0\}$ and $a_2=v_i$ for some $i \in [\tilde{n}-1]$, then $N((a_1,a_2))\sse C_{u_{i-1}}\cup C_{u_i}$ in which case $(a_1,a_2),(b_1,b_2)\in W_{u_{i'}}$ for $i' \in \{i-1,i\}$. Similarly, if $a_2 \in K_{k+1}\bs \{v_0\}$ and $a_1=v_i$ for some $j \in [\tilde{n}-1]$, then $(a_1,a_2),(b_1,b_2)\in W_{w_{j'}}$ for $j' \in \{j-1,j\}$. Finally, if $a_1=v_i$ and $a_2=v_j$ for some $i,j \in [\tilde{n}-1]$, then $N(a_1,a_2)\sse L_{w_{\ell-1}}\cup L_{w_{\ell}}$ and thus $(a_1,a_2),(b_1,b_2)\in W_{u_{\ell'}}$ for some $\ell' \in [\ell-1,\ell]$ where $\ell:=\min{\{i,j\}}$. As such, every edge is in a bag. 

	It remains to show that for any $(a_1,a_2) \in V(G_{k,n}\boxtimes G_{k,n})$, the subtree $T^{(a_1,a_2)}:=T[t \in V(T):(a_1,a_2)\in W_t]$ is connected. If $(a_1,a_2)\in X$, then $T^{(a_1,a_2)}=V(T)$. If $a_1,a_2 \in V(K_k)\bs\{v_0\}$, then $V(T^{(a_1,a_2)})=\{y\}$. If $a_1\in V(K_k)\bs\{v_0\}$ and $a_2=v_i$ for some $i \in [\tilde{n}-1]$, then $V(T^{(a_1,a_2)})=\{u_{i-1},u_i\}$. If $a_1=v_j$ for some $j \in [\tilde{n}-1]$ and $a_2\in V(K_k)\bs\{v_0\}$, then $V(T^{(a_1,a_2)})=\{w_{j-1},w_j\}$. If $a_1=v_j$ and $a_2=v_i$ for some $i,j \in [\tilde{n}-1]$ then $V(T^{(a_1,a_2)})=\{z_{\ell-1},z_{\ell}\}$ where $\ell:=\min{\{i,j\}}$. In each case, $T^{(a_1,a_2)}$ is connected. Therefore, $(T,\mathcal{W})$ is indeed a tree decomposition of $G_{k,n}\boxtimes G_{k,n}$ with width $O(n+k^2)$ as required.
\end{proof}

We now consider lower bounds for the treewidth of cartesian and strong products. \citet{wood2013treewidth} established the following lower bound for highly-connected graphs.
\begin{thm}[\cite{wood2013treewidth}]\label{twCartProdLower}
	For all $k$-connected graphs $G_1$ and $G_2$ each with at least $n$ vertices,
	$$\tw(G_1\boxtimes G_2)\geq \tw(G_1 \CartProd G_2)\geq k(n-2k+2)-1.$$
\end{thm}

We now present two new lower bounds. For a graph $G$ and $\epsilon \in [\frac{2}{3},1)$, a partition $(A,S,B)$ of $V(G)$ is an \emph{$\epsilon$-separation} if $1\leq |A|,|B|\leq \epsilon |V(G)|$ and there is no edge between $A$ and $B$. The \emph{order} of $(A,S,B)$ is $|S|$. \citet{robertson1986algorithmic} showed that graphs with small treewidth have separations with small order.
\begin{lem}[\cite{robertson1986algorithmic}]\label{separatortw}
	For every $\epsilon \in [\frac{2}{3},1)$, every graph with treewidth $k$ has an $\epsilon$-separation of order at most $k+1$.
\end{lem}

\begin{lem}\label{SeptwProductNEW}
	For all $\epsilon,\beta \in [\frac{2}{3},1)$ where $\beta >\epsilon$ and integers $k,n,m\geq 1$ where $m\geq kn$ and $n \geq \frac{1}{1-\beta}$ and for all connected graphs $G$ and $H$ with $m$ and $n$ vertices respectively such that every $\beta$-separation of $G$ has order at least $k$, every $\epsilon$-separation of $G \CartProd H$ has order at least $(1-\frac{\epsilon }{\beta}) kn$.
\end{lem}

\begin{proof}
	Let $V(H):=[n]$. Let $(A,S,B)$ be a $\epsilon$-separation of $G \CartProd H$. Our goal is to show that $|S|\geq (1-\frac{\epsilon }{\beta}) kn$. For $i \in [n]$, let $G^{(i)}$ be the copy of $G$ in $G \CartProd H$ induced by $\{(v,i):v\in V(G)\}$. We say that $G^{(i)}$ \emph{belongs} to $A$ if $|A\cap V(G^{(i)})|\geq \beta m$, and $G^{(i)}$ \emph{belongs} to $B$ if $|B\cap V(G^{(i})|\geq \beta m$.
	
	Suppose that some copy $G^{(i)}$ belongs to $A$ and some copy $G^{(j)}$ belongs to $B$. Let $X:=\{v\in V(G):(v,i)\in A, (v,j)\in B\}$. Thus $|X|\geq (2\beta-1) m>0,$ which implies $i \neq j$. Since $H$ is connected, for each $x \in X$ there is a path from $(x,i)$ to $(x,j)$ contained within the subgraph of $G \CartProd H$ induced by $\{(x,\ell):\ell \in V(H)\}$. Since these paths are pairwise disjoint and each path contains a vertex from $S$, we have $|S|\geq |X|\geq (2\beta-1)kn\geq (1-\frac{\epsilon }{\beta}) kn$.
	
	Now assume, without loss of generality, that no copy of $G$ belongs to $B$. Say $t$ copies of $G$ belong to $A$. Then $\beta tm \leq |A|\leq \epsilon nm$, implying that $t \leq \frac{\epsilon n}{\beta}$. Thus, at least $(1-\frac{\epsilon}{\beta})n$ copies of $G$ belong to neither $A$ nor $B$. Now consider such a copy $G^{(i)}$. If $G^{(i)}$ is contained in $S\cup A$, then 
	$|S\cap V(G^{(i)})|\geq (1-\beta)m\geq k.$
	Similarly, if $G^{(i)}$ is contained in $S\cup B$, then $|S\cap V(G^{(i)})|\geq k$. Otherwise, $G^{(i)}$ contains vertices in both $A$ and $B$, in which case $(A\cap V(G^{(i)}),S\cap V(G^{(i)}),B\cap V(G^{(i)}))$ is a $\beta$-separation of $G^{(i)}$ and thus, $|S\cap V(G^{(i)})|\geq k$. Therefore,
	$$|S|=\sum_{i=1}^n|V(G^{(i)})\cap S|\geq (1-\tfrac{\epsilon }{\beta}) kn$$
	as required.
\end{proof}
\Cref{separatortw,SeptwProductNEW} imply the following.

\begin{thm}\label{SeparatorstwProduct}
		For all $\epsilon,\beta \in [\frac{2}{3},1)$ where $\beta >\epsilon$ and integers $k,n,m\geq 1$ where $m\geq kn$ and $n \geq \frac{1}{1-\beta}$ and for all connected graphs $G$ and $H$ with $m$ and $n$ vertices respectively such that every $\beta$-separation of $G$ has order at least $k$, 
			$$\tw(G \CartProd H)\geq (1-\tfrac{\epsilon }{\beta})kn-1.$$
\end{thm}

 By applying \Cref{SeptwProductNEW,twStrongProduct}, we determine the treewidth of $d$-dimensional grid graphs up to a constant factor.

\begin{cor}
	For fixed $d\geq 2$ and all integers $n_1\geq \dots \geq n_d\geq 1$, 
	\begin{equation*}
			\tw(P_{n_1}\CartProd\dots \CartProd P_{n_d})=\Theta \left( \prod_{j=2}^{d}n_j \right)
		\quad\text{and}\quad 
			\tw(P_{n_1}\boxtimes\dots \boxtimes P_{n_d})=\Theta \left(\prod_{j=2}^{d}n_j\right).
	\end{equation*}
\end{cor}
\begin{proof}[Proof Sketch]
	We proceed by induction on $d\geq 2$. For the upper bound, apply \Cref{twStrongProduct} by setting $G_1:=P_{n_1}\boxtimes\dots \boxtimes P_{n_{d-1}}$ and $G_2:=P_{n_d}$ with the induction hypothesis that $\tw(P_{n_1}\boxtimes\dots \boxtimes P_{n_d})\leq \prod_{j=2}^{d}n_j$ for all $d \geq 2$. For the lower bound, apply \Cref{SeptwProductNEW} by setting $G:=P_{n_1}\CartProd\dots \CartProd P_{n_{d-1}}$ and $H:=P_{n_d}$ with the induction hypothesis that for all $d \geq 2$ and $\epsilon \in [\frac{2}{3},1)$, there is a constant $c(d,\epsilon)$ such that for all sufficiently large $n_d$ (as a function of $d$ and $\epsilon)$, every $\epsilon$-separation in $P_{n_1}\CartProd\dots \CartProd P_{n_d}$ has order at least $c(d,\epsilon)\prod_{j=2}^{d}n_j$. The lower bound then follows by \Cref{separatortw}.
\end{proof}

This result demonstrates that $d$-dimensional grids are a family of graphs for which \Cref{SeparatorstwProduct} is tight up to a constant factor, whereas the lower bound given by \Cref{twCartProdLower} is not of the correct order for this family.

We now present another lower bound in terms of the Hadwiger number. For a graph $G$, the \emph{Hadwiger number} of $G$, $\eta(G)$, is the maximum integer $t$ such that $K_t$ is a minor of $G$. 
\begin{thm}\label{twHadwiger}
	For all graphs $G$ and $H$,
	$$\tw(G\boxtimes H)\geq \eta(H)(\tw(G)+1)-1.$$
\end{thm}
\begin{proof}
	Let $(B_i)_{i \in V(K_t)}$ be a model of $K_t$ in $H$ where $t:=\eta(H)$. By \Cref{BrambleDuality}, there is a bramble $\mathcal{B}$ in $G$ of such that every hitting set of $\mathcal{B}$ has order at least $\tw(G)+1$. For $X \in \mathcal{B}$ and $i \in [t]$, let $(X,i):=\{(v,u)\in V(G\boxtimes H):v \in X,u\in B_i\}$. Let $G_{X,i}$ denote the subgraph of $G \boxtimes H$ that is induced by $(X,i)$. Since $G[X]$ and $H[B_i]$ are connected, $G_{X,i}$ is connected. For $i\in [t]$, let $\mathcal{B}_i:=\{(X,B_i):X \in \mathcal{B}\}$. Then $\mathcal{B}_i$ is a bramble for $G \boxtimes (H[B_i])$.
	
	Let $\mathcal{C}:=\{(X,B_i):X \in \mathcal{B}, i \in [t]\}$. Consider $(X,B_i),(Y,B_j)\in \mathcal{C}$. Since $X$ and $Y$ touch in $G$, for some vertices $v \in X$ and $w \in Y$, either $v=w$ or $vw \in E(G)$. Moreover, there exists $u_i \in B_i$ and $u_j \in B_j$ such that $u_i=u_j$ (if $i=j)$ or $u_iu_j\in E(H)$. Thus, in $G \boxtimes H$, the vertices $(v,u_i)$ and $(w,u_j)$ are adjacent or equal. Since $(v,u_i)\in (X,B_i)$ and $(v,u_j)\in (Y,B_j)$, the sets $(X,B_i)$ and $(Y,B_j)$ touch. Hence $\mathcal{C}$ is a bramble in $G \boxtimes H$. Let $J$ be a hitting set for $\mathcal{C}$. For each $i \in [t]$, let $J_i:=\{(v,u) \in J:u \in B_i\}$. Thus $J_i$ is a hitting set for the bramble $\mathcal{B}_i$ (in $G \boxtimes B_i)$. By \Cref{BrambleDuality}, $|J_i|\geq \tw(G)+1$. Since the $J_i$'s are pairwise disjoint, $|J|\geq t(\tw(G)+1)$. Hence $\tw(G \boxtimes H)\geq t(\tw(G)+1)-1$ by \Cref{BrambleDuality}.
\end{proof}

\Cref{twHadwiger,twStrongProduct} together imply that for every graph $G$, 
$$\tw(G\boxtimes K_n)=(\tw(G)+1)n-1.$$

The next two theorems characterise when the cartesian and strong products have bounded treewidth and when they have bounded pathwidth. 

\begin{thm}\label{twCartStrongProd}
	The following are equivalent for monotone graph classes $\G_1$ and $\G_2$:
	\begin{compactenum}
		\item $\G_1 \boxtimes \G_2$ has bounded treewidth;
		\item $\G_1 \CartProd \G_2$ has bounded treewidth;
		\item both $\G_1$ and $\G_2$ have bounded treewidth and $\tilde{\textsf{v}}(\G_1)$ or $\tilde{\textsf{v}}(\G_2)$ is bounded; and
		\item $\G_1 \boxtimes \G_2$ has bounded Hadwiger number;
	\end{compactenum} 
\end{thm}
\begin{proof}
	Building on the work of \citet{wood2011clique}, \citet{pecaninovic2019complete} showed that $(3)$ and $(4)$ are equivalent. Since treewidth is closed under subgraphs, $(1)$ implies $(2)$. By \Cref{twStrongProduct}, $(3)$ implies $(1)$. To show that $(2)$ implies $(3)$, suppose that $\G_1 \CartProd \G_2$ has bounded treewidth. If either $\G_1$ or $\G_2$ have unbounded treewidth, then $\G_1 \CartProd \G_2$ has unbounded treewidth, a contradiction. If the connected graphs in both $\G_1$ and $\G_2$ have unbounded order, then by \Cref{twCartProdLower}, $\G_1 \CartProd \G_2$ has unbounded treewidth, a contradiction. As such, both $\G_1$ and $\G_2$ have bounded treewidth and $\tilde{\textsf{v}}(\G_1)$ or $\tilde{\textsf{v}}(\G_2)$ is bounded.
\end{proof}

We omit the proof for the following theorem as it is identical to \Cref{twCartStrongProd}.
\begin{thm}\label{pwCartStrongProd}
	The following are equivalent for monotone graph classes $\G_1$ and $\G_2$:
	\begin{compactenum}
		\item $\G_1 \boxtimes \G_2$ has bounded pathwidth;
		\item $\G_1 \CartProd \G_2$ has bounded pathwidth; and
		\item both $\G_1$ and $\G_2$ have bounded pathwidth and $\tilde{\textsf{v}}(\G_1)$ or $\tilde{\textsf{v}}(\G_2)$ is bounded.
	\end{compactenum} 
\end{thm}


We conclude this subsection with some open problems. First and foremost, can we determine the treewidth of $G \boxtimes H$ up to a constant factor? A first step in resolving this question is to determine if \Cref{twHadwiger} can be strengthened by showing that there exists a constant $c>0$ such that $\tw( G \boxtimes H ) \geq c \tw(G) \tw(H)$ for all graphs $G$ and $H$. Similar questions arise for the cartesian product. It is also open whether there exist a constant $c>0$ such that $\tw(G \boxtimes H) \leq c \tw(G \CartProd H)$ for all graphs $G$ and $H$.

\subsection{Direct Product with $K_2$}\label{SectionDirectK2}
We now consider when a direct product has bounded treewidth and when it has bounded pathwidth. In comparison with the other two products, characterisation for this product is more involved. Before considering the direct product of two classes of graphs, we first investigate the direct product of a class of graphs with a single graph. \Cref{twStrongProduct} immediately provides an upper bound for the treewidth and pathwidth of the direct product of a class of graphs with a single graph. The challenge therefore lies in proving a lower bound. In related results, \citet{bottreau1998some} demonstrated that for every graph $G$ and non-bipartite graph $H$, $G$ is a minor of $G \times H$, and hence $\tw(G \times H)\geq \tw(G)$ and $\pw(G \times H)\geq \pw(G)$. This lower bound for treewidth was independently shown by \citet{eppstein2020treewidth}. It remains to consider the case when $H$ is bipartite. A case of particular interest is when $H=K_2$. \Citet{thomassen1988disjoint} showed the following. 

\begin{thm}[\cite{thomassen1988disjoint}]\label{MaintwBipartiteSubgraph}
	There exists a function $f$ such that every graph $G$ with $\tw(G)\geq f(k)$ contains a bipartite subgraph $\hat{G}\sse G$ such that $\tw(\hat{G})\geq k$.
\end{thm}
Note that the function $f$ in \Cref{MaintwBipartiteSubgraph} is exponential with respect to $k$. Since every bipartite subgraph of $G$ is also a subgraph of $G \times K_2$, we have the following.
\begin{cor}\label{twUnboundedKtwo}
	For every graph class $\G$ with unbounded treewidth, $\G \times K_2$ also has unbounded treewidth.
\end{cor}

Together with \Cref{twStrongProduct}, \Cref{twUnboundedKtwo} demonstrates that $\tw(G)$ and $\tw(G \times K_2)$ are tied. Similarly, for pathwidth, as explained in our companion paper \cite{HW2021bipartite}, an analogous result to \Cref{MaintwBipartiteSubgraph} is implied by the excluded forest minor theorem \cite{robertson1983graph,bienstock1991quickly}.

\begin{thm}[\cite{HW2021bipartite}]\label{MainpwBipartiteSubgraph}
	There exists a function $f$ such that every graph $G$ with $\pw(G)\geq f(k)$ contains a bipartite subgraph $\hat{G}\sse G$ such that $\pw(\hat{G})\geq k$.
\end{thm}

\begin{cor}\label{GtimesKTwoUnboundedPathwidth}
	For every graph class $\G$ with unbounded pathwidth, $\G \times K_2$ also has unbounded pathwidth.
\end{cor}

Together with \Cref{twStrongProduct}, it follows that $\pw(G)$ and $\pw(G \times K_2)$ are tied. Note that the function $f$ in \Cref{MainpwBipartiteSubgraph} is exponential with respect to $k$. We now strengthen these results by demonstrating that $\tw(G)$ and $\tw(G \times K_2)$ are polynomially tied and then showing that this implies that $\pw(G)$ and $\pw(G \times K_2)$ are polynomially tied.

\begin{thm}\label{TwDirectKTwoPolynomial}
	There exists a positive constant $c$ such that $\tw(G \times K_2)\geq k$ for every graph $G$ with $\tw(G)\geq ck^4\log^{5/2}(k)$.
\end{thm}

To prove \Cref{TwDirectKTwoPolynomial}, we make use of grid-like-minors which were introduced by \citet{reed2012polynomial}. A \emph{grid-like-minor of order $\ell$} in a graph $G$ is a set $\Path$ of paths in $G$ such that the intersection graph of $\Path$ is bipartite and contains a $K_{\ell}$ minor.\footnote{The intersection graph of a set $X$, whose elements are sets, has vertex set $X$ where distinct vertices are adjacent whenever the corresponding sets have a non-empty intersection.} \citet{reed2012polynomial} showed the following.

\begin{thm}[\cite{reed2012polynomial}]\label{MainGridLikeMinors}
	For some positive constant $c$, every graph with treewidth at least $c\ell^4\sqrt{\log(\ell)}$ contains a grid-like-minor of order $\ell$. 
\end{thm}

We now explain how to adapt the proof for \Cref{MainGridLikeMinors} to show that if a graph $G$ has sufficiently large treewidth, then $G \times K_2$ contains a grid-like-minor of large order. A key lemma that Reed and Wood used to prove \Cref{MainGridLikeMinors} is the following.

\begin{lem}[\cite{reed2012polynomial}]\label{BrambleImpliesDisjointPaths}
	For all integers $k,\ell \geq 1$, every graph $G$ with treewidth at least $k\ell-1$ contains $\ell$ disjoint paths $P_1,\dots,P_{\ell}$, and for distinct $i,j \in [\ell]$, $G$ contains $k$ disjoint paths between $P_i$ and $P_j$.
\end{lem}

Given a black-white colouring of the vertices of a bipartite graph, to \emph{switch} the colouring means to recolour the black vertices white and the white vertices black. For a graph $G$ with vertex-disjoint subgraphs $H_1$ and $H_2$, a path $P=(v_1,\dots,v_n)$ in $G$ \emph{joins} $H_1$ and $H_2$ if $V(P) \cap V(H_1)=\{v_1\}$ and $V(P) \cap V(H_2)=\{v_n\}$. We make use of the following lemma from our companion paper \cite{HW2021bipartite} which generalises the result by Erd\H{o}s \cite{erdos1965some} which states that every graph contains a bipartite subgraph with at least half of the edges.

\begin{lem}[\cite{HW2021bipartite}]\label{MainLemmaConstructBipartiteSubgraphs}
	Let $G$ be a graph and let $H_1,\dots,H_t$ be vertex-disjoint bipartite subgraphs in $G$. Suppose there exists a set of internally disjoint paths $\Pcal=\{P_1,\dots, P_k\}$ in $G$ such that for all $i \in [k]$, $P_i$ joins $H_j$ and $H_{\ell}$ for some distinct $j,\ell \in [t]$ and that $P_i$ and $H_b$ are disjoint for all $b \in [t]\bs \{j,\ell\}$. Then there exists a subset $\mathcal{S} \sse \Pcal$ where $|\mathcal{S}|\geq k/2$ such that the graph $\hat{G}:=(\bigcup_{i \in [t]} H_i) \cup (\bigcup_{S \in \mathcal{S}}S)$ is a bipartite subgraph of $G$.
\end{lem}

The next lemma is the main original contribution for this subsection in which we adapt \Cref{BrambleImpliesDisjointPaths} to the setting of a direct product with $K_2$.
\begin{lem}\label{BrambleImpliesDisjointPathsKTwo}
	Let $G$ be a graph with treewidth at least $2k\ell-1$ for some integers $k,l \geq 1$. Then there exists a set $X \sse E(K_{\ell})$ where $|X|\geq |E(K_{\ell})|/2$ such that $G \times K_2$ contains $\ell$-disjoint paths $P_1, \dots, P_{\ell}$ and for all $ij \in X$, $G \times K_2$ contains $k$-disjoint paths between $P_i$ and $P_j$.
\end{lem}

\begin{proof}
	By \Cref{BrambleImpliesDisjointPaths}, $G$ contains $\ell$ disjoint paths $\tilde{P}_1,\dots,\tilde{P}_{\ell}$, and a set $\mathcal{Q}_{i,j}$ of $2k$ disjoint paths between $\tilde{P}_i$ and $\tilde{P}_j$ where $i,j \in [\ell]$ are distinct. Take a proper black-white colouring of $\tilde{P}_1,\dots,\tilde{P}_j$. For a path $Q=(v_1^{(i,j)},\dots,v_n^{(i,j)}) \in \mathcal{Q}_{i,j}$, we say that it is \emph{agreeable} if there exists a black-white colouring of $Q$ such that the colour of $v_1$ and $v_n$ corresponds with the colour they were previously assign. Otherwise, we say that $Q$ is \emph{disagreeable}. Observe that if we switch the black-white colouring of $\tilde{P}_j$, then every agreeable path becomes disagreeable and every disagreeable path becomes agreeable. If there are more agreeable paths than disagreeable, then let $\mathcal{M}_{i,j}$ be the set of agreeable paths, otherwise let $\mathcal{M}_{i,j}$ be the set of disagreeable path. By the pigeon-hole principle, $|\mathcal{M}_{i,j}|\geq k$. Note that the set $\mathcal{M}_{i,j}$ of paths are \emph{pairwise-agreeable}, in the sense that if one path in $\mathcal{M}_{i,j}$ is agreeable then all the paths in $\mathcal{M}_{i,j}$ are agreeable. 
	
	For each distinct $i,j \in [\ell]$, let $R_{i,j}=(v_1^{(i,j)},\dots,v_{n_{i,j}}^{(i,j)})$ be an arbitrary path in $\mathcal{M}_{i,j}$. We now define an auxiliary graph $J$ as follows. Let $J$ consist of the $\ell$-disjoint paths $\tilde{P}_1,\dots, \tilde{P}_\ell$. For each distinct $i,j \in [\ell]$, add a path $\tilde{R}_{i,j}$ from $v_1^{(i,j)} \in V(\tilde{P}_i)$ to $v_{n_{i,j}}^{(i,j)} \in V(\tilde{P}_j)$ of length $n_{i,j}$ that is internally disjoint from all other vertices in $J$. Let $\mathcal{R}=\{\tilde{R}_{i,j}: i,j \in [\ell], i \neq j\}$. By \Cref{MainLemmaConstructBipartiteSubgraphs}, there exists a set $\mathcal{S}\sse \mathcal{R}$ where $|\mathcal{S}|\geq |\mathcal{R}|/2$ such that $\tilde{J}=(\bigcup_{i \in [\ell]} \tilde{P}_i)\cup \mathcal{S}$ is bipartite. Let $X=\{ij:R_{i,j} \in \mathcal{S}\}$ and note that $|X|\geq |E(K_{\ell})|/2$. Now for all $ij \in X$, since the set of paths $\mathcal{M}_{i,j}$ are pairwise agreeable, it follows that for all $Q=(v_1,\dots,v_n) \in \mathcal{M}_{i,j}$, we can add a path $\hat{Q}$ of length $|Q|$ from $v_1 \in V(\tilde{P}_i)$ to $v_n \in V(\tilde{P}_j)$ to the graph $\tilde{J}$ that is internally disjoint to all other paths without compromising $\tilde{J}$ being bipartite. Do this for all paths in $\mathcal{M}_{i,j}$ that has not yet been considered whenever $ij \in X$ and let $\hat{J}$ be the bipartite graph obtained. Let $\phi: V(\hat{J}) \to [2]$ be a proper $2$-colouring of $\hat{J}$.
	
	Now consider $G \times K_2$. For each $i \in [\ell]$, let $P_i$ be the path in $G \times K_2$ induced by $\{(v,\phi(v)): v \in V(\tilde{P}_i)\}$. Since $\tilde{P}_{i}$ and $\tilde{P}_j$ are disjoint for distinct $i,j \in [\ell]$, it follows that $P_1\dots,P_\ell$ are $\ell$-disjoint paths in $G \times K_2$. It remains to show that for all distinct $ij \in X$, there exists $k$ disjoint paths from $P_i$ to $P_j$. Let $Q=(v_1^{(i,j)},\dots,v_n^{(i,j)}) \in \mathcal{M}_{i,j}$. By construction of $\hat{J}$, for each $v \in V(\hat{Q})$ there exists a corresponding $\hat{v}\in V(\tilde{Q})$. Let $Y$ be the path in $G \times K_2$ that is induced by $\{(v,\phi(\hat{v})): v \in V(Q)\}$ and let $\mathcal{Y}_{i,j}$ be the set of such paths. Since the set of paths $\mathcal{M}_{i,j}$ are internally-disjoint, it follows that the set of paths $\mathcal{Y}_{i,j}$ are also internally-disjoint. Moreover, $|\mathcal{Y}_{i,j}|=|\mathcal{M}_{i,j}|\geq k$ as required. 
\end{proof}

Let $d(\ell)$ be the minimum integer such that every graph with no $K_{\ell}$ minor is $d(\ell)$-degenerate. \citet{kostochka1982mean, kostochka1984average} and \citet{thomason1984contraction} independently proved that $d(\ell)\in \Theta(\ell\sqrt{\log(\ell)})$.

\begin{thm}[\cite{kostochka1984average,thomason1984contraction,thomason2001extremal}]\label{AvgDegForceMinor}
	Every graph that does not contain $K_{\ell}$ as a minor is $d(\ell)$-degenerate where $d(\ell):= c{\ell}\sqrt{\log(\ell)}$ and thus has average degree at most $2c{\ell}\sqrt{\log(\ell)}$ where $c$ is a positive constant.
\end{thm}

The proof for the following theorem is a simple adaption of the proof of \Cref{MainGridLikeMinors} in \cite{reed2012polynomial}.

\begin{thm}\label{DirectKTwoGridLikeMinor}
	For every graph $G$ with treewidth at least $c\ell^4\log^{5/2}(\ell)$, $G \times K_2$ contains a grid-like-minor of order $\ell$ for some constant $c$.
\end{thm}

\begin{proof}
	Let $t:=\ceil{2(2c\ell\sqrt{\log(\ell)}+1)}$ and $k:=\ceil{4e\binom{t}{2}d(t)}$ where $c$ is specified by \Cref{AvgDegForceMinor}. Let $G$ be a graph with treewidth at least $c\ell^4\log^{5/2}(\ell)$ which is at least $2kt-1$ for an appropriate value of $c$. By \Cref{BrambleImpliesDisjointPathsKTwo}, there exists a set $X \sse E(K_{t})$ where $|X|\geq |E(K_{t})|/2$ such that $G \times K_2$ contains $t$-disjoint paths $P_1, \dots, P_{t}$ and for all $ij \in X$, $G \times K_2$ contains a set $\mathcal{Q}_{i,j}$ of $k$-disjoint paths between $P_i$ and $P_j$. Let $J$ be the subgraph of $K_t$ with vertex set $V(K_t)$ and edge set $X$. By \Cref{AvgDegForceMinor}, $J$ contains $K_{\ell}$ as a minor. From here on in, the rest of the proof follows \cite[Theorem 1.2]{reed2012polynomial}.
	
	For distinct $ij,ab \in X$, let $H_{i,j,a,b}$ be the intersection graph of $\mathcal{Q}_{i,j} \cup \mathcal{Q}_{a,b}$. Since $H_{i,j,a,b}$ is bipartite, if $K_{t}$ is a minor of $H_{i,j,a,b}$ then we are done. So assume that $K_{t}$ is not a minor of $H_{i,j,a,b}$. By \Cref{AvgDegForceMinor}, $H_{i,j,a,b}$ is $d(t)$-degenerate.
	Let $H$ be the intersection graph of $\cup\{\mathcal{Q}_{i,j}:ij \in X\}$. Then $H$ is $|X|$-colourable where each colour class is some $\mathcal{Q}_{i,j}$. Since each colour class has $k$ vertices, and each pair of colour-classes in $H$ induce a $d(t)$-degenerate subgraph, then by \cite[Lemma~4.3]{reed2012polynomial}, $H$ has an independent set with one vertex from each colour class. That is, in each set $\mathcal{Q}_{i,j}$ there is a path $Q_{i,j}$ such that $Q_{i,j}\cap Q_{a,b}=\emptyset$ for distinct $ij,ab \in X$. Consider the set of paths $\Path:=\{P_i:i \in [t]\} \cup \{Q_{i,j}:ij \in X\}.$
	The intersection graph of $\Path$ is bipartite and contains the $1$-subdivision of $J$. Since $J$ contains $K_{\ell}$ as a minor, $\Path$ is a grid-like-minor of order $\ell$ in $G$.
\end{proof}

\citet{reed2012polynomial} showed that the treewidth of a grid-like-minor of order $\ell$ is at least $\floor{\ell/2}-1$. As such, \Cref{DirectKTwoGridLikeMinor} implies \Cref{TwDirectKTwoPolynomial}, and $\tw(G)$ and $\tw(G \times K_2)$ are polynomially tied. 

We now show that \Cref{TwDirectKTwoPolynomial} implies a polynomial lower bound for the pathwidth of $G \times K_2$. \citet{groenland2021approximating} proved the following.
\begin{thm}[\cite{groenland2021approximating}]\label{LargeCBTPathwidth}
	For all integers $t,h\geq 1$, for every graph $G$ with $\tw(G)\geq t-1$, either $\pw(G)\leq th+1$ or $G$ contains a subdivision of a complete binary tree of height $h$.
\end{thm}
It is well-known that the pathwidth of a subdivision of a complete binary tree with height $2k$ is at least $k$ (see \cite{scheffler1989die}). Therefore, \Cref{LargeCBTPathwidth,TwDirectKTwoPolynomial} imply the following.

\begin{thm}\label{PwDirectKTwoPolynomial}
	There exists some positive constant $c$ such that $\pw(G \times K_2)\geq k$ for every graph $G$ with $\pw(G)\geq ck^5\log^{5/2}(k)$.
\end{thm}

By \Cref{twStrongProduct,PwDirectKTwoPolynomial}, we conclude that $\pw(G)$ and $\pw(G \times K_2)$ are polynomially tied.

\subsection{Direct Product}\label{SectionDirectProduct}
In this subsection, we characterise when the direct product of two classes of graphs has bounded treewidth (\Cref{twDirectProductGeneralGraphs}) and when it has bounded pathwidth (\Cref{pwDirectProductGeneralGraphs}). We begin with some definitions. For every integer $k \geq 0$, the $k$-daddy-longlegs $W^{(k)}$ is the tree with $V(W^{(k)})=\{r,u_1,\dots,u_k,v_1,\dots,v_k\}$ and $E(W^{(k)})=\{ru_i,u_iv_i: i \in [k]\}$; see \Cref{fig:kclaw}. Let $G$ be a graph. The \emph{daddy-longlegs number of $G$}, $\dll(G)$, is the maximum integer $k\geq 0$ such that $W^{(k)}$ is a minor of $G$. The \emph{path number of $G$}, $\pathnumber(G)$, is the maximum integer $n\geq 0$ such that $G$ contains a path on $n$ vertices. Clearly the path number of a graph is at least its diameter. A set $A\sse V(G)$ is a \emph{vertex cover} of $G$ if $G-A$ is an independent set. The \emph{vertex cover number of $G$}, $\tau(G)$, is the minimum size of a vertex cover of $G$.

 \begin{figure}[!htb]
	\begin{center}
		\includegraphics[width=0.2\linewidth]{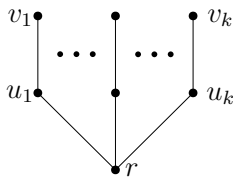}
		\caption{The $k$-daddy-longlegs $W^{(k)}$.}
		\label{fig:kclaw}
	\end{center}
\end{figure}

We now work towards showing that graphs with bounded daddy-longlegs number and bounded path-number have bounded vertex cover number. We begin with some basic lemmas.
 
\begin{lem}\label{TreeLeavesDiameter}
	For all integers $j,n\geq 1$, if a tree $T$ has at most $j$ leaves, then $\textsf{v}(T)\leq \ceil{\frac{j}{2}\pathnumber(T)}$.
\end{lem}
\begin{proof}
	We proceed by induction on $n$. For $n=1$, the claim holds trivially. Now suppose the claims holds up to $n-1$. Let $T$ be a tree with at most $j$ leaves and $\pathnumber(T)\leq n$. Let $T_L$ be the subtree of $T$ obtained by deleting its leaves. Then $T_L$ has at most $j$ leaves.  If $T$ contains a path $P$ with $n$ vertices, then the endpoints of $P$ are leaves in $T$. If $T$ contains a path $P$ with $n-1$ vertices, then at least one of the endpoints of $T$ is a leaf. As such, $\pathnumber(T_L)\leq n-2$. By induction, $\textsf{v}(T_L)\leq \ceil{j(n-2)/2}$. Since at most $j$ vertices were deleted from $T$ to obtain $T_L$, we have $\textsf{v}(T)\leq \textsf{v}(T_L)+j\leq \ceil{jn/2}$ as required.
\end{proof}

The next lemma is important for our upcoming characterisations.
\begin{lem}\label{TreewidthConditions}
	For every connected graph $G$, $\tau(G)\leq \ceil{(\dll(G)+1)\, \pathnumber(G)/2}.$
\end{lem}

\begin{proof}
	Let $T$ be a DFS spanning tree of $G$ rooted at some $v \in V(G)$. Note that $v$ may be a leaf vertex in $T$. Let $L$ be the set of leaves of $T$ excluding $v$ and let $T_L:=T-L$. Since $T$ is a DFS tree, $L$ is an independent set, and hence $V(T_L)$ is a vertex cover of $G$. We now bound $\textsf{v}(T_L)$. Let $\tilde{L}:=\{\tilde{u}_1,\dots,\tilde{u}_{j}\}$ be the leaves in $T_L$ excluding $v$. If $j=0$ then $V(T_L)=\{v\}$ and we are done. Otherwise, for every $i \in [j]$, let $\tilde{v}_i\in L$ where $\tilde{u}_i\tilde{v}_i \in E(T)$. Note that $\tilde{v}_i\neq \tilde{v}_j$ whenever $i \neq j$. Let $T_r:=T_L-\tilde{L}$ and let $\mu$ be a model of $W^{(j)}$ in $G$ defined by $\mu(r)=T_r$, $\mu(u_i)=\tilde{u}_i$ and $\mu(v_i)=\tilde{v}_i$ for all $i \in [j]$. Then $W^{(j)}$ is a minor of $G$ and hence $j \leq \dll(G)$. Thus $T_L$ has at most $j+1$ leaves. Since $\pathnumber(T_L)\leq \pathnumber(G)$, by \Cref{TreeLeavesDiameter}, $\textsf{v}(T_L)\leq \ceil{(\dll(G)+1)\pathnumber(G)/2}$ as required.
\end{proof}


We now present several lemmas for direct products in the general framework of $H$-decompositions. The motivation for doing so is that once results are established within this framework, we can quickly deduce bounds for both treewidth and pathwidth. As such, this framework provides a unified approach for which we can tackle both parameters at once. 

The \emph{square} of a graph $G$, denoted $G^2$, is the graph with $V(G^2)=V(G)$ where $uv \in E(G^2)$ if $\dist_G(u,v)\leq 2$. For our purposes, the key property of this graph is that $N_G[v]$ is a clique in $G^2$ for every vertex $v \in V(G)$. The next basic lemma concerning $G^2$ is useful.
\begin{lem}\label{HDecompGSquared}
	For all graphs $G$ and $H$ where $G$ has maximum degree $\Delta$ and has an $H$-decomposition with width at most $k$, $G^2$ has an $H$-decomposition with width at most $(k+1)(\Delta+1)-1$.
\end{lem}

\begin{proof}
	Let $(H,\W)$ be an $H$-decomposition of $G$ with width at most $k$. For all $t \in V(H)$, let $B_t:=\bigcup_{v \in W_t} N_G[v]$ and let $\B:=\{B_t:t\in V(H)\}$. We claim that $(H,\B)$ is an $H$-decomposition of $G^2$ with width at most $(k+1)(\Delta+1)-1$. Note that for all $v \in V(G^2)$ there exists a node $t_1 \in V(H)$ such that $v \in W_{t_1}$. By construction, $W_{t_1} \sse B_{t_1}$ and hence $v \in B_{t_1}$. Hence every vertex in $G^2$ is in a bag. Let $uv \in E(G^2)$. Now if $uv$ is also an edge in $E(G)$ then there exists a node $t_2 \in V(H)$ such that $\{u,v\}\sse W_{t_2}$ and hence $\{u,v\}\sse B_{t_2}$. Otherwise, $\dist_G(u,v)=2$ and hence $u$ and $v$ share a common neighbour $y \in V(G)$ in $G$. Now since $(H,\W)$ is an $H$-decomposition, there exists a node $t_3 \in V(H)$ such that $y \in W_{t_3}$. So by construction, $\{u,v\}\sse N_G[y]\sse B_{t_3}$, and hence the endpoints of each edge in $G^2$ is in a bag. 
	
	It remains to show that for all $v \in V(G^2)$, the induced subgraph $H[\{t:v\in B_t\}]$ is connected. For all $v \in V(G^2)$, let $H^{(v)}:=H[\{t:v\in W_t\}]$. By construction, 
	\begin{equation*}
		H[\{t:v\in B_t\}]=H[\{t: N[v]\cap W_t \neq \emptyset\}]=H[\bigcup_{u \in N[v]}\{t :u \in W_t\}]=\bigcup_{u \in N[v]}H^{(u)}.
	\end{equation*}
	 Now for all $u \in N_G(v)$, there exists $x \in V(H)$ such that $\{u,v\}\in W_x$ and as such, $V(H^{(v)})\cap V(H^{(u)})\neq \emptyset$. Hence $H[\{t:v\in B_t\}]$ is a connected subgraph of $H$ for all $v \in V(G^2)$. We conclude that $(H,\B)$ is indeed an $H $-decomposition of $G^2$. Note that the width of $(H,\B)$ is at most $(k+1)(\Delta+1)-1$ since for every vertex in a bag of $(H,\W)$, we add at most $\Delta$ other vertices to that bag in order to obtain $(H,\B)$.
\end{proof}

The next lemma is the heart of our characterisation of when a direct product has bounded treewidth or when it has bounded pathwidth. For the sake of simplicity, we shall consider a path to be a subdivision of $K_1$.

\begin{lem}\label{SufficientCondition}
	For all connected graphs $G_1$, $G_2$ and $H$ where $\tau(G_1)=t$, and $G_2$ has maximum degree $\Delta$ and has an $H$-decomposition with width at most $k$, there exists an $H'$-decomposition of $G_1 \times G_2$ with width at most $t(k+1)(\Delta+1)$ where $H'$ is a subdivision of $H$.
\end{lem}
\begin{proof}
	Let $A\sse V(G_1)$ be a vertex cover of $G_1$ where $|A|= t$. We may assume that $A$ is a clique in $G_1$. Let $\tilde{G}_1=G[A]$ and let $L:=V(G_1)-A$. By \Cref{HDecompGSquared}, there exists an $H$-decomposition of $G_2^2$ with width at most $(k+1)(\Delta+1)-1$. By \Cref{HDecompStrongProduct}, there exists an $H$-decomposition $(H,\W)$ of $\tilde{G}_1 \boxtimes G_2^2$ that has width at most $t(k+1)(\Delta+1)-1$.
	
	If $L=\emptyset$, then $G_1 \boxtimes G_2 \sse \tilde{G}_1 \times G_2^2$ and we are done. Otherwise, let $(\ell,v) \in L \times V(G_2)$. Now $N_{G_1}(\ell)$ is a clique in $K_t$ and $N_{G_2}(v)$ is a clique in $G_2^2$. Thus, $N_{G_1\times G_2}((\ell,v))$ is a clique in $\tilde{G}_1\boxtimes G_2^2$. As such, there exists a node $x \in V(H)$ such that $N_{G_1\times G_2}((\ell,v)) \sse W_{x}$. We now explain how to modify $(H,\W)$ to obtain an $H'$-decomposition $(H',\B)$ of $G_1 \times G_2$ where $H'$ is a subdivision of $H$.
	
	Let $(\ell_1,v_1),(\ell_2,v_2),\dots,(\ell_j,v_j)$ be an arbitrary ordering of the vertices in $L \times V(G_2)$ where $j=|L \times V(G_2)|$. For each $i \in [j]$, let $x_i \in V(H)$ be a vertex where $N_{G_1\times G_2}((\ell,v)) \sse W_{x_i}$. Initialise $i:=1$, $H_i:=H$, $V_i:=V(\tilde{G}_1 \boxtimes G_2^2)$ and $\W_i:=\W$. For $i=1,2,\dots,j+1$, apply the following procedure: let $y_i$ be a neighbour of $x_i$ in $H_i$. Subdivide the edge $x_iy_i \in E(H_i)$ and let $H_{i+1}$ be the graph obtained and let $z_i$ be the new vertex from the subdivided edge (see \Cref{fig:SufficientTwSubdivision}). Note that if no such $x_iy_i$ edge exists then $H_i=K_1$. In which case, let $H_{i+1}$ be obtained from $H_i$ by adding the vertex $z_i$ and edge $x_iz_i$. Let $W_{z_i}=W_{x_i} \cup \{(\ell_i,v_i)\}$ and $\W_{i+1}=\W_i\cup \{W_{z_i}\}$ and $V_{i+1}=V_i \cup \{(\ell_i,v_i)\}$. 

	 \begin{figure}[!htb]
		\begin{center}
			\includegraphics[width=0.37\linewidth]{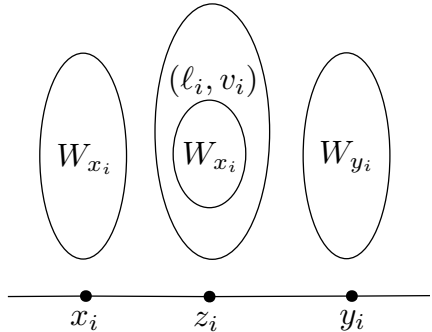}
			\caption{Subdividing the $x_iy_i$ edge in $H_i$ to obtain $H_{i+1}$.}
			\label{fig:SufficientTwSubdivision}
		\end{center}
	\end{figure}
	
	Once the above procedure has completed, let $H':=H_{j+1}$ and for all $x \in V(H')$, let $B_x:=W_x$ and $\B:=\{B_x:x \in V(H')\}$. Note that $H'$ is a subdivision of $H$ since $H_{i+1}$ is obtained by subdividing an edge of $H_i$ for all $i \in [j]$.
	
	We now demonstrate that $(H',\B)$ is an $H'$-decomposition of $G_1 \times G_2$. Let $(u_1,u_2)\in V(G_1 \times G_2)$. If $u_1 \in V(G_1)-L$, then $(u_1,u_2)\in V(\tilde{G}_1 \boxtimes G_2^2)$ and hence there is a $x \in V(H)$ such that $(u_1,u_2)\in W_x$. By the construction of $(H',\B)$, it follows that $(u_1,u_2)\in B_x$. Otherwise, $u_i \in L$ in which case $(u_1,u_2)=(l_i,v_i)$ for some $i \in [j]$ and $(u_1,u_2) \in W_{z_i}$. Hence every vertex is in a bag. 
	
	Let $(u_1,u_2)(u_3,v_4)\in E(G_1 \times G_2)$. Since $L$ is an independent set, at most one of $u_1$ and $u_3$ is in $L$. Suppose that $u_1 \in L$. Then there exists $i \in [j]$ such that $(u_1,u_2)=(\ell_i,v_i)$. Since $(u_3,u_4) \in N_{G_1\times G_2}[(\ell_i,v_i)]$, by the construction of $B_{z_i}$ it follows that $(u_1,u_2),(u_3,v_4)\in B_{z_i}$. Now if neither $u_1$ nor $u_3$ is in $L$, then there is a node $x \in V(H)$ such that $(u_1,u_2),(u_3,u_4)\in W_x$ and hence $(u_1,u_2),(u_3,u_4) \in B_x$. Hence, the endpoints of each edge is in a bag. 
	
	It remains to show that for all $(u_1,u_2) \in V(G_1 \times G_2)$, the subgraph $H'[\{x\in V(H'):(u_1,u_2) \in B_x\}]$ is connected. We prove the following claim by induction on $i$: 
	
	\textbf{Claim:} For every vertex $(u_1,u_2)\in V_i$, the induced subgraph $H_i[\{x\in V(H'):(u_1,u_2)\in W_x\}]$ is connected. 
	
	For $i=1$, this claim holds since $(H_1,\W_1)$ is an $H$-decomposition of $\tilde{G}_1 \boxtimes G_2^2$. Now assume that it holds up to $i-1$ where $i\geq 2$. The graph $H_{i}$ is obtained from $H_{i-1}$ by replacing the edge $x_iy_i$ by the path $x_i,z_i,y_i$. Furthermore, $V_i\bs V_{i-1}=\{(\ell_{i},v_{i})\}$. Now the vertex $(\ell_i,v_i)$ is only in the bag $W_{z_i}$ and hence $H_i[x:(\ell_{i},v_{i})\in W_x]$ is connected. Now if $(u_1,u_2)\in V_{i-1}$ then by induction $H_{i-1}[x:(u_1,u_2)\in W_x]$ is connected. For the sake of contradiction, suppose that $H_{i}[x:(u_1,u_2)\in W_x]$ is disconnected. The only way for this to occur is if $(u_1,u_2)\in W_{x_i}$ and $(u_1,u_2) \not \in W_{z_i}$. However, by the construction of $W_{z_i}$, we have $W_{x_i} \sse W_{z_i}$, a contradiction. Hence $H_{i}[\{t:(u_1,u_2)\in W_t\}]$ is connected for all $(u_1,u_2)\in V_i$.
	
	This completes our proof that $(H',\B)$ is an $H'$-decomposition of $G_1 \times G_2$ where $H'$ is a subdivision of $H$. We conclude by noting that the width of this $H'$-decomposition is at most $1$ more than the width of $(H,\W)$ as required. 
\end{proof}

When the graph $H$ in \Cref{SufficientCondition} realises the treewidth or pathwidth of $G_2$, we have the following.

\begin{cor}\label{twSufficientCondition}
	For all connected graphs $G_1$ and $G_2$,
	\begin{align*}
		\tw(G_1 \times G_2)&\leq \tau(G_1)(\tw(G_2)+1)(\Delta(G_2)+1), \text{ and}\\
		\pw(G_1 \times G_2)&\leq \tau(G_1)(\pw(G_2)+1)(\Delta(G_2)+1).
	\end{align*}
\end{cor}

The next two lemmas are key cases for when a direct product has unbounded treewidth. The first lemma follows by a result in our companion paper \cite{HW2021hadwiger} and the well-known fact that $\tw(K_{n,n})=n$.

\begin{lem}[\cite{HW2021hadwiger}]\label{ClawPathHadwiger}
	For every integer $k\geq 1$ and graph $G$ with $\dll(G)\geq k$, we have
	$K_{k,k}$ is a minor of $G \times P_{2k}$ and thus $\tw(G \times P_{2k})\geq k$.
\end{lem}

The next lemma is folklore and uses the well-known fact that $\tw(P_n \CartProd P_n)=n$ \cite{bodlaender1998partial}.
\begin{lem}\label{twFamilyGraphsUnbounded}
	For integers $b,n \geq 1$, let $S_b$ be the star with $b$ leaves and $P_n$ be the path on $n$ vertices. Then $\tw(S_b \times S_b)\geq b$ and $\tw(P_{2n-1} \times P_{2n-1})\geq n$.
\end{lem}

\begin{proof}
	 For the direct product of stars, $K_{b,b}\sse S_b \times S_b$ by \Cref{CompleteSubgraphsDirectProduct} and hence $\tw(S_b \times S_b)\geq b$. 
	\begin{figure}[!htb]
		\begin{center}
			\includegraphics[width=0.4\linewidth]{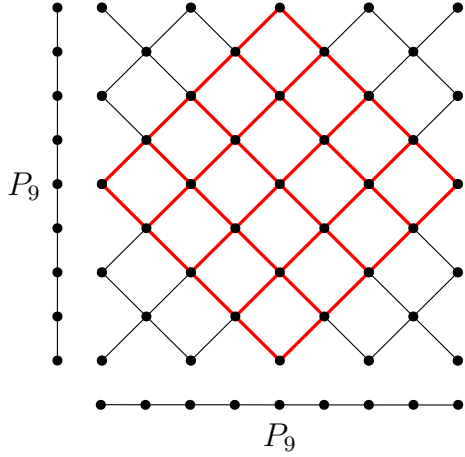}
			\caption{The $(5 \times 5)$-grid in $P_9 \times P_9$.}
			\label{fig:PtimesP}
		\end{center}
	\end{figure}

	Now consider the direct product of two paths. Since $\tw(P_n \CartProd P_n)\geq n$, it suffice to show that $P_n \CartProd P_n \sse P_{2n-1} \times P_{2n-1}$. Let $V(P_{2n-1})=[2n-1]$. For each $i \in [n]$, let $H^{(i)}$ be the subgraph of $P_{2n-1} \times P_{2n-1}$ induced by $\{(i-k,n+i-k):k\in [n]\}.$	For each $j \in [n]$, let $P^{(j)}$ be the subgraph of $P_{2n-1} \times P_{2n-1}$ induced by $\{(j+k,n-j+k):k\in [n]\}$.	Then $(\bigcup_{i \in [n]} H^{(i)})\cup (\bigcup_{j \in [n]}P^{(j)})$ defines a $P_n \CartProd P_n$ subgraph in $P_{2n-1} \times P_{2n-1}$ where for all $i,j \in [n]$, $H^{(i)}$ is the $i^{th}$-horizontal path in $P_n \CartProd P_n$ and $P^{(j)}$ is the $j^{th}$-vertical path in $P_n \CartProd P_n$, as required (see \Cref{fig:PtimesP}).
\end{proof}

The next lemma is the Moore bound; see \cite{miller2012moore} for a survey.

\begin{lem}\label{GeneralBoundedDegreeBoundedDiam}
For every connected graph $G$ with maximum degree $\Delta>1$ and diameter $d$, 	
	\begin{equation*}
		\textsf{v}(G)\leq 
		\begin{cases}
			1+\Delta \frac{(\Delta-1)^d-1}{\Delta-2}, &\text{if } \Delta>2,\\
			2d+1, & \text{if } \Delta=2.\\
		\end{cases}
	\end{equation*}
\end{lem}

For a graph $G$, let $\tilde{\tau}(G)$ be the maximum vertex cover number of a component of $G$. We now prove our characterisation for when a direct product has bounded treewidth.
\begin{thm}\label{twDirectProductGeneralGraphs}
	Let $\G_1$ and $\G_2$ be monotone graph classes that contains $K_2$. Then $\tw(\G_1\times \G_2)$ is bounded if and only if $\tw(\G_1)$ and $\tw(\G_2)$ are bounded and at least one of the following conditions hold:
	\begin{compactitem}
		\item $\tilde{\textsf{v}}(G)$ or $\tilde{\textsf{v}}(G)$ is bounded;
		\item $\tilde{\tau}(G)$ and $\Delta({\G}_2)$ are bounded; or
		\item $\tilde{\tau}(G)$ and $\Delta({\G}_1)$ are bounded.
	\end{compactitem}
\end{thm}

\begin{proof}
	Assume that $\tw(\G_1 \times \G_2)$ is bounded. By \Cref{twUnboundedKtwo}, if $\G_1$ or $\G_2$ has unbounded treewidth then $\G_1 \times \G_2$ has unbounded treewidth. Hence, we may assume that there exists an integer $k \geq 1$ such that $\tw(\G_1)\leq k$ and $\tw(\G_2) \leq k$. Now suppose there exists an integer $c_1 \geq 1$ such that $\tilde{\textsf{v}}({\G}_1) \leq c_1$ or $\tilde{\textsf{v}}({\G}_2) \leq c_1$. By \Cref{twStrongProduct}, $\tw(\G_1 \times \G_2)\leq (k+1)c_1-1$ and we are done. 
	
	It remains to consider the case when neither $\tilde{\textsf{v}}({\G}_1)$ nor $\tilde{\textsf{v}}({\G}_2)$ is bounded. Suppose that $\G_1$ and $\G_2$ both have unbounded maximum degree. Then for every integer $b\geq 1$, the star $S_b$ is a member of both $\G_1$ and $\G_2$ and $(S_b \times S_b)_{b \in \N}\sse \G_1 \times \G_2$. However, by \Cref{twFamilyGraphsUnbounded}, this is a contradiction since $(S_b \times S_b)_{b \in \N}$ has unbounded treewidth. Therefore either $\G_1$ or $\G_2$ has bounded maximum degree. 
	
	Without loss of generality, there exists an integer $c_2\geq 1$ such that $\Delta(\G_2)\leq c_2$. Since $\tilde{\textsf{v}}({\G}_2)$ is unbounded, by \Cref{GeneralBoundedDegreeBoundedDiam} it follows that $\pathnumber({\G}_2)$ is unbounded. Now if $\pathnumber({\G}_1)$ is also unbounded, then $(P_n \times P_n)_{n \in \N}$ is a family of graphs in ${\G}_1 \times {\G}_2$. This contradicts $\tw(\G_1 \times \G_2)$ being bounded since by \Cref{twFamilyGraphsUnbounded}, $(P_n \times P_n)_{n \in \N}$ has unbounded treewidth. Similarly, if $\dll(\G_1)$ is unbounded, then by \Cref{ClawPathHadwiger}, $\tw(\G_1 \times \G_2)$ is unbounded. So assume there exists integers $j,n \geq 1$ such that for every connected graph $G\in \G_1$, we have $\dll(G)\leq j$ and $\pathnumber(G)\leq n$. By \Cref{TreewidthConditions}, $\hat{\tau}({\G}_1)\leq \ceil{(j+1)n/2}$. By \Cref{twSufficientCondition}, $\tw(\G_1 \times \G_2)\leq \ceil{(j+1)n/2}(k+1)(\Delta+1)$ and thus, is bounded. As we have considered all possibilities, this completes our proof.
\end{proof}

The next theorem characterises when a direct product has bounded pathwidth. We omit the proof as it is identical to \Cref{twDirectProductGeneralGraphs} except we use \Cref{GtimesKTwoUnboundedPathwidth} instead of \Cref{twUnboundedKtwo}.

\begin{thm}\label{pwDirectProductGeneralGraphs}
	Let $\G_1$ and $\G_2$ be monotone graph classes that contains $K_2$. Then $\G_1 \times \G_2$ has bounded pathwidth if and only if $\G_1$ and $\G_2$ both have bounded pathwidth and at least one of the following holds:
	\begin{compactitem}
		\item $\textsf{v}(\hat{\G}_1)$ or $\textsf{v}(\hat{\G}_2)$ is bounded;
		\item $\tau(\hat{\G}_1)$ and $\Delta({\G}_2)$ are bounded; or
		\item $\tau(\hat{\G}_2)$ and $\Delta({\G}_1)$ are bounded.
	\end{compactitem}
\end{thm}

\fontsize{10}{11} 
\selectfont 
	
\let\oldthebibliography=\thebibliography
\let\endoldthebibliography=\endthebibliography
\renewenvironment{thebibliography}[1]{%
	\begin{oldthebibliography}{#1}%
		\setlength{\parskip}{0.3ex}%
		\setlength{\itemsep}{0.3ex}%
}{\end{oldthebibliography}}
	
\bibliographystyle{DavidNatbibStyle}
\bibliography{RobReferences}

\end{document}